\documentclass{amsart}
\usepackage{amssymb,graphicx}
\usepackage{mathrsfs}
\usepackage{color}
\usepackage{enumerate}
\usepackage{a4wide}
\usepackage[bbgreekl]{mathbbol}
\usepackage[pdfpagelabels,colorlinks,citecolor=blue,urlcolor=blue,bookmarks=false,hypertexnames=true,backref]{hyperref}

\newcommand{\Rl}{\mathbb{R}}
\newcommand{\Cplx}{\mathbb{C}}

\newcommand{\Ntrl}{\mathbb{N}}
\newcommand{\Circ}{\mathbb{T}}

\newcommand{\Bc}{\mathcal{B}}

\newcommand{\Ec}{\mathcal{E}}

\newcommand{\Hc}{\mathcal{H}}
\newcommand{\Ic}{\mathcal{I}}

\newcommand{\Kc}{\mathcal{K}}
\newcommand{\Lc}{\mathcal{L}}

\newcommand{\Sc}{\mathcal{S}}

\newcommand{\Uc}{\mathcal{U}}

\newcommand{\tr}{\mathrm{tr}}

\newcommand{\Tr}{\mathrm{Tr}}

\newcommand{\re}{\mathrm{re}}

\newcommand{\Res}{\mathrm{Res}}
\newcommand{\tf}{\mathfrak{t}}
\newcommand{\Tb}{\mathbb{T}}
\newcommand{\res}{\mathrm{res}}
\newcommand{\Psib}{\mathbb{\Psi}}
\newcommand{\Sigmab}{\mathbb{\Sigma}}
\newcommand{\Pb}{\mathbb{P}}
\newcommand{\Qb}{\mathbb{Q}}
\newcommand{\Ab}{\mathbb{A}}
\newcommand{\Bb}{\mathbb{B}}

\newcommand{\dhom}{Q}
\newcommand{\kbb}{\mathbb{k}}
\newcommand{\fbb}{\mathbb{f}}

\newcommand{\evz}{\mathrm{tr}}
\newcommand{\dens}{\Omega}
\newcommand{\Beta}{\mathrm{B}}
\newcommand{\Gexp}{\mathbb{E}\!\mathrm{xp}}
\newcommand{\phg}{{\mathrm{phg}}}

\newcommand{\Int}{{\mathrm{Int}}}

\def\XXint#1#2#3{{\setbox0=\hbox{$#1{#2#3}{\int}$ }
\vcenter{\hbox{$#2#3$ }}\kern-.6\wd0}}

\numberwithin{equation}{section}

\newtheorem{theorem}{Theorem}[section]
\newtheorem{proposition}[theorem]{Proposition}
\newtheorem{corollary}[theorem]{Corollary}
\newtheorem{definition}[theorem]{Definition}
\newtheorem{lemma}[theorem]{Lemma}

\newtheorem{conjecture}{Conjecture}

\theoremstyle{remark}

\newtheorem{remark}[theorem]{Remark}

\title{Connes' trace theorem and the log-polyhomogeneous calculus for Carnot manifolds}
\date{\today}
\author{E.~McDonald}
\address{Universit\"at Bonn, Germany}
\email{ed@eamcdonald.xyz}

\begin{document}
\maketitle{}

\begin{abstract}
    The Wodzicki residue is the unique trace on the algebra of classical pseudodifferential operators on a closed manifold, and Connes in 1988 proved that it coincides with the Dixmier trace. There are also ``higher" residues, defined on the set of operators whose symbols that can be expressed as polynomials of a logarithm, introduced by Lesch, and these are related to other singular traces. A Carnot manifold is a manifold $M$ whose tangent bundle $TM$ is equipped with a nested family $H$ of sub-bundles $H_0\leq H_1 \leq \cdots \leq TM$ which defines a filtration of the Lie algebra of vector fields on $M.$ Associated to a Carnot manifold is a pseudodifferential calculus $\Psi_H(M),$ which measures sections of $H_k$ as having order $k.$
    Recently, Dave-Haller and Couchet-Yuncken proposed definitions of a residue functional on the algebra of pseudodifferential operators adapted to a Carnot manifold. We prove that Connes' trace theorem holds in this setting. We also introduce an analogy of Lesch's log-polyhomogeneous calculus for Carnot manifolds, define the corresponding higher residues, and give their spectral description in terms of singular traces.
\end{abstract}

\section{Introduction}
This paper concerns several issues around the pseudodifferential calculus on a Carnot manifold $(M,H),$ as defined by van Erp and Yuncken \cite{vanErpYuncken2019}. Our main goal is to show that Connes' trace formula holds in their setting (Theorem \ref{main_theorem} below). Our secondary goal is to develop a log-polyhomogeneous extension of their calculus, including a corresponding Connes' trace theorem for this extension.

The set of classical pseudodifferential operators is related to polyhomogeneous functions, and one of the key insights of \cite{vanErpYuncken2019} was that polyhomogeneous functions can be characterised as the restrictions of approximately homogeneous functions on a higher-dimensional space. To illustrate, let $\{\delta_t\}_{t>0}$ be a group of dilations on $\Rl^N.$ That is, $\delta_t(e_j) = t^{w_j}e_j$ for $j=1,\ldots,N,$ where $e_j$ is the $j$th basis vector and $w_1,\ldots,w_N$ are positive numbers. Say that $f \in C^\infty(\Rl^N)$ is \emph{approximately homogeneous} of order $m\in \Cplx$ if for all $t>0$ we have
\[
    f-t^{-m}f\circ \delta_t \in \Sc(\Rl^N)
\]
where $\Sc(\Rl^N)$ is the Schwartz space. According to a theorem of Taylor \cite[Proposition 2.2]{Taylor1984}, the following are equivalent:
\begin{itemize}
    \item{} $f$ is approximately homogeneous of order $m,$
    \item{} $(Z-m)f \in \Sc(\Rl^N),$ where $Zf := \frac{d}{dt}f
    |_{t=1}\circ \delta_t$ is vector field generator of $\delta,$ i.e.
    \[
        Z = \sum_{j=1}^N w_jx_j\frac{\partial}{\partial x_j},
    \]
    \item{} $f$ has the form
    \[
        f = h+g
    \]
    where $g\in \Sc(\Rl^N),$ and $h\in C^\infty(\Rl^N)$ is homogeneous away from zero in the sense that $h(\delta_t\xi)=t^mh(\xi)$ for all $t>1$ and all sufficiently large $\xi$
    \item{} and (when $\Re(m)<0$), there exists $a\in \Sc(\Rl^N)$ such that
    \[
        f = \int_0^1 t^{-m-1}a\circ \delta_t\,dt.
    \]
\end{itemize}
Recall that a distribution $f$ is homogeneous of order $m$ if and only if $f$ satisfies the Euler equation $(Z-m)f=0.$ The above equivalences tell us a function $f$ is approximately homogeneous if and only if it \emph{approximately} satisfies the Euler equation, in the sense that $(Z-m)f\in \Sc(\mathbb{R}^N).$ The function $a$ in the fourth item is simply $a = (Z-m)f.$

On the other hand we say that a smooth function $f$ is \emph{polyhomogeneous} of order $m$ if there exist homogeneous functions $\{f_j\}_{j=0}^\infty,$ where $f_j$ has order $m-j,$ such that
\[
    f \sim \sum_{j=0}^\infty f_j
\]
in the sense that $f(\xi)-\sum_{j=0}^n f_j(\xi)$ decays faster than $(1+|\xi|)^{\Re(m)-n}$ as $|\xi|\to\infty.$ It was Debord and Skandalis who observed that the following are equivalent \cite[Theorem 3.7]{DebordSkandalis2014}:
\begin{itemize}
    \item{} $f$ is polyhomogeneous of order $m$
    \item{} There exists a smooth function $\widetilde{f}$ on $\Rl^N\times \Rl$ which is approximately homogeneous with respect to the dilation $\widetilde{\delta}_t(\xi,h) = (\delta_t\xi,th)$ such that
    \[
        f(\xi) = \widetilde{f}(\xi,1),\quad \xi\in \Rl^N
    \]
    \item{} and (when $\Re(m)<0$) there exists $a\in \Sc(\Rl^N\times \Rl)$ such that
    \[
        f(\xi) = \int_0^1 t^{-m-1}a(\delta_t\xi,t)\,dt.
    \]
\end{itemize}
Note that the function in the third item can recovered from $f$ by $a = (\widetilde{Z}-m)\widetilde{f}$ where $\widetilde{Z}$ is the generator of $\widetilde{\delta}_t.$
Let $S^m_{\phg}(\Rl^N)$ be the set of polyhomogeneous functions of order $m$ on $\Rl^N,$ and let $\dhom$ be the homogeneous dimension, defined as $\dhom = \frac{\log(\det(\delta_t))}{\log(t)}=w_1+\cdots+w_N.$ Two important functionals on $S^{m}_{\phg}(\Rl^N)$ are the \emph{canonical integral}
\[
    \Int:S^{m}_{\phg}(\Rl^N)\to \Cplx,\quad m\notin -Q+\{0,1,2,3,\ldots\}
\]
and the \emph{residue}
\[
    \Res:S^{-\dhom}_{\phg}(\Rl^N)\to \Cplx.
\]
The canonical integral extends the ordinary (Lebesgue) integral, which is initially defined for operators of order $m$ with $\Re(m)<-Q$ and extended to all those operators whose order does not belong to $-\dhom+\Ntrl$ by a Hadamard finite part regularisation. The residue of $f \in S^{-\dhom}_{\phg}(\Rl^N)$ can be given as the integral over a ``unit sphere" of the order $-\dhom$ homogeneous part $f_0.$
Both $\Int(f)$ and $\Res(f)$ have elementary descriptions in terms of an approximately homogeneous extension $\widetilde{f}.$ Let $\widetilde{f}$ be an order $m$ approximately homogeneous function on $\Rl^{N+1}$ with $\widetilde{f}(\xi,1)=f(\xi),$ and denote
\[
    a(\xi,h) := (\widetilde{Z}-m)\widetilde{f}(\xi,h) = \frac{d}{dt}\big|_{t=1}t^{-m}\widetilde{f}(\delta_t\xi.th),\quad (\xi,h)\in \Rl^{N}\times \Rl.
\]
When the order $m$ of $f$ is $-Q,$ the residue is given by the formula
\begin{equation}\label{residue_commutative_formula}
    \Res(f) = \int_{\Rl^N} a(\xi,0)\,d\xi
\end{equation}
while if $m\notin -Q+\Ntrl,$ then the canonical integral is
\begin{align*}
    \Int(f) &= \sum_{j=0}^{n-1} \frac{1}{j!}\frac{1}{j-m-\dhom}\int_{\Rl^N} a^{(j)}(\xi,0)\,d\xi\\
            &\quad + \frac{1}{(n-1)!}\int_0^1\int_{\Rl^N} a^{(n)}(\xi,u)\,u^{n-m-\dhom-1}\Beta(1-u;n,m+\dhom-n+1)\,d\xi du
\end{align*}
where $n$ is suitably large so that the integral converges, and $a^{(j)}$ denotes the $j$th derivative of $a$ in its second argument. Here,
\[
    \Beta(x;\alpha,\beta) = \int_0^x t^{\alpha-1}(1-t)^{\beta-1}\,dt
\]
is the incomplete beta function.

A classical pseudodifferential operator on $\Rl^d$ of order $m\in \Cplx$ is defined by a symbol function $\sigma$ on $\Rl^d\times \Rl^d$ which is polyhomogeneous in the second variable. That is, $\sigma$ admits an asymptotic expansion
\[
    \sigma(x,\xi) \sim \sum_{k=0}^\infty \sigma_k(x,\xi)
\]
where $\sigma_k(x,\xi)$ is approximately homogeneous of order $m-k$ in $\xi$ with respect to the standard isotropic dilations on $\Rl^d.$ The Schwartz kernel of a pseudodifferential operator on $\Rl^d$ with symbol $\sigma$ is given by the Fourier transform
\[
    K(x,y) = (2\pi)^{-d}\int_{\Rl^d}\exp(i(x-y,\xi))\sigma(x,\xi)\,d\xi.
\]
This is an oscillatory integral, which should be understood as defining a distribution in $(x,y).$ The \emph{canonical trace} and \emph{noncommutative residue} of a pseudodifferential operator $T$ with symbol $\sigma$ are given by the canonical integral and the residue of $\sigma,$ i.e.
\[
    \mathrm{TR}(T) := (2\pi)^{-d}\int_{\Rl^d} \Int(\sigma(x,\cdot))\,dx,\quad \Res(T) := (2\pi)^{-d}\int_{\Rl^d} \Res(\sigma(x,\cdot))\,dx.
\]
This is meaningful provided that the symbol $\sigma$ has the correct order ($m\notin -d+\Ntrl$ for the canonical trace, and $m=-d$ for the noncommutative residue) and also also assuming that $\sigma$ decays suitably in $x$ so that the integrals over $x$ converges.

Van Erp and Yuncken used the above characterisations of polyhomogeneous functions in terms of almost-homogeneous extensions to characterise the Schwartz kernels of classical pseudodifferential operators on a manifold in terms of approximately homogeneous distributions on the tangent groupoid.
It follows from the above stated results that $K(x,y)$ is the Schwartz kernel of a classical pseudodifferential operator of order $m$ if and only if there exists a distribution $\widetilde{K}$ on $\Rl^d\times \Rl^d\times \Rl,$ smooth in the first and third variables, such that $K(x,y) = \widetilde{K}(x,y,1)$ and
\[
    \widetilde{K}(x,x+z,h)-t^{-m-d}\widetilde{K}(x,x+t^{-1}z,th)
\]
is smooth in $(x,z,h)$ for every $t>0.$ Van Erp and Yuncken phrased this in terms of the tangent groupoid: if $\Rl^d$ is replaced by a manifold $M,$ then the Schwartz kernel $K$ of an operator is a distribution on $M\times M$ and the kernels of pseudodifferential operators are those distributions which admit an approximately homogeneous extension $\widetilde{K}$ to the tangent groupoid $\Tb M.$
An advantage of their theory is that it readily generalises to define pseudodifferential operators associated to a nontrivial filtration of the tangent bundle, also called a Carnot structure.

Couchet--Yuncken \cite{CouchetYuncken2024} and Mohsen \cite{Mohsen2024residue} gave tangent groupoid descriptions of the noncommutative residue of a pseudodifferential operator. In the case of Couchet--Yuncken, the noncommutative residue was extended to Carnot manifolds. They proved that their residue was equivalent to one earlier defined by Dave--Haller \cite{DaveHaller2020}, and conjectured that there should be a form of Connes' trace theorem for that setting (we give more details on the trace theorem in Subsection \ref{ctt_subsection} below).

The main purpose of this paper is to explain why Connes' trace formula is indeed true for calculus associated to a Carnot manifold, and why it is that formulas like \eqref{residue_commutative_formula} make the proof almost trivial.

The second purpose of this paper concerns the characterisation of log-polyhomogeneous pseudodifferential operators in terms of extension of kernels to the tangent groupoid. A smooth function $f$ on a $\Rl^N$ equipped with dilations $\{\delta_t\}_{t>0}$ is said to be $k$-log-polyhomogeneous of order $m\in \Cplx$ if $f$ admits an asymptotic expansion
\[
    f(\xi) \sim \sum_{\ell=0}^{k} \sum_{j=0}^{\infty} \log(|\xi|)^{\ell}f_{j,\ell}(\xi)
\]
where $f_{j,\ell}$ is homogeneous of order $m-j$ with respect to $\delta,$ and $|\cdot|$ is some norm on $\Rl^N$ which satisfies $|\delta_t\xi|=t|\xi|.$ This definition is due to Lesch \cite{Lesch1999}.
It is not particularly difficult to prove, but nonetheless apparently not previously noticed, that $k$-log-polyhomogeneous functions admit a description similar to that of Debord and Skandalis for polyhomogeneous functions. The following are equivalent:
\begin{itemize}
    \item{} $f$ is $k$-log-polyhomogeneous,
    \item{} There exists a smooth function $\widetilde{f}$ on $V\times \Rl$ satisfying $(\widetilde{Z}-m)^{k+1}\widetilde{f}\in \Sc(V\times \Rl)$ and $f(\xi) = \widetilde{f}(\xi,1),$
    \item{} (when $\Re(m)<0$) there exists $a \in \Sc(V\times \Rl)$ such that
    \[
        f(\xi) = \frac{1}{k!}\int_0^1 (-\log t)^{k}t^{-m-1}a(\delta_t\xi,t)\,dt.
    \]
\end{itemize}
We prove this in Appendix \ref{characterisation_appendix}. Accordingly, the kernels of $k$-log-polyhomogeneous pseudodifferential operators on a manifold $M$ can be characterised by their extensions to the tangent groupoid $\mathbb{T}M.$ By analogy we can define $k$-log-polyhomogeneous pseudodifferential operators associated to a Carnot structure. Lesch defined a canonical integral and a $k$th order residue for $k$-log-polyhomogeneous functions, and it turns out that these admit descriptions in terms of the function $a$ similar to those given above. Log-polyhomogeneous pseudodifferential operators are are defined as pseudodifferential operators with symbol function $\sigma(x,\xi)$ that are log-polyhomogeneous in $\xi.$ Analogous to van Erp and Yuncken's characterisation of classical pseudodifferential operators by extension to the tangent groupoid, we can characterise log-polyhomogeneous pseudodifferential operators in the same way. Hence we propose a version of $k$-log-polyhomogeneous pseudodifferential operators on a Carnot manifold. There is an obvious choice of canonical trace, and $k$th order residue, for these operators. It turns out that traces of holomorphic families of $k$-log-polyhomogeneous pseudodifferential operators have meromorphic continuations with poles of order at most $k+1,$ and this has consequences for the distribution of eigenvalues of log-polyhomogeneous pseudodifferential operators.

\section{Background material}
\subsection{Carnot manifolds}\label{carnot_intro_section}
There is some variation in the terminology concerning Carnot manifolds, which have also been called filtered manifolds \cite{vanErpYuncken2017,vanErpYuncken2019}. The situation here is the same as that called the ``group germ case" by Goodman \cite{Goodman-lnm-1976}.

Let $M$ be a manifold. For a vector bundle $E\to M,$ we denote $C^\infty(M,E)$ for the space of smooth sections of $E.$ In this note we will adapt the following terminology: a Lie filtration of $TM$ is a nested sequence
$H = \{H^j\}_{j=0}^\infty$ of sub-bundles of the tangent bundle $TM$ with depth $N\geq 1$
\[
    0 = H^0 \subset H^1 \subset \cdots \subset H^N = TM
\]
satisfying the following property: if $X\in C^\infty(M, H^j)$ and $Y \in C^\infty(M,H^k),$ then $[X,Y] \in C^\infty(M, H^{j+k}),$ where $H^{j}=TM$ when $j>N.$ The {trivial filtration} has $N=1.$ The pair $(M,H)$ is called a {Carnot manifold}. Important special cases the trivial filtration, which has $N=1$ and Heisenberg manifolds.

The associated graded bundle $\tf_HM$ is defined as the direct sum $\bigoplus_{k=1}^N \tf_HM^k,$ where
\[
    \tf_HM^k := H^k/H^{k-1},\quad k\geq 1.
\]
It can be easily checked with the Leibniz rule that Lie bracket of vector fields descends to a Lie bracket on the fibres of $\tf_HM,$ making each fibre into a graded nilpotent Lie algebra of dimension $\mathrm{dim}(M).$ The fibrewise exponential of $\tf_HM$ is a bundle of nilpotent Lie groups called the osculating groupoid, denoted $T_HM.$ The fibres of $T_HM$ are called the osculating groups of $(M,H).$ The homogeneous dimension of $(M,H)$ is defined as
\[
    \dhom := \sum_{k=1}^N k\cdot\mathrm{rank}(\tf_HM^k) = \sum_{k=1}^N \mathrm{rank}(H^k).
\]
The {dilation} action of $\Rl^{\times}$ on $\tf_HM$ defined as
\[
    \delta_t = \bigoplus_{k=1}^N t^{k},\quad t > 0
\]
where the $k$th summand acts on $\tf_HM^k.$ The determinant of $\delta_t$ as an $\mathrm{End}(\tf_HM)$-valued function is $t^{\dhom}.$

A \emph{splitting} is a smooth bundle isomorphism
\[
    \psi:\tf_HM\to TM
\]
such that the restriction of $\psi$ to the summand $\tf_HM^k$ is right-inverse to the quotient map $H^k\to \tf_HM^k.$ We can always choose a splitting $\psi,$ and if $\psi_1,\psi_2$ are two splittings, then the composition $\psi_2^{-1}\psi_1:\tf_HM\to TM$ is upper diagonal in the sense that it maps $\tf_{H}M^k$
into $\bigoplus_{j\leq k} \tf_HM^j.$
A splitting $\psi$ induces a linear isomorphism
\[
    \Uc(\tf_HM)\to \mathrm{DO}(M)
\]
where $\Uc(\tf_HM)$ is the universal enveloping algebra of the Lie algebra of sections of $\tf_HM,$ and $\mathrm{DO}(M)$ is the ring of differential operators on $M.$ Denote the differential operators of degree $m\geq 1$ (with respect to the grading on $\tf_HM$) by $\mathrm{DO}_H^m(M).$
This gives $\mathrm{DO}(M)$ the structure of a filtered algebra. That is to say,
\[
    \mathrm{DO}(M) = \bigcup_{m\geq 1} \mathrm{DO}_H^m(M),\quad \mathrm{DO}^m_H(M)\cdot \mathrm{DO}^n_H(M) \subseteq \mathrm{DO}^{m+n}_{H}(M),\quad m,n\geq 1.
\]
We can identify the quotient space $\mathrm{DO}^m_H(M)/\mathrm{DO}^{m-1}_H(M)$ with the space $\Uc_m(\tf_HM)$ of elements of $\Uc(\tf_HM)$ that are homogeneous of degree $m$ with respect to $\delta.$ This can be thought of as an exact sequence
\[
    0\rightarrow \mathrm{DO}_H^{m-1}(M)\to \mathrm{DO}^{m}_H(M)\to \Uc_m(\tf_HM)\to 0.
\]
Many authors have considered the problem of extending $\mathrm{DO}(M)$ to a filtered algebra
\[
    \Psi_H(M) =\bigcup_{m\in \Rl} \Psi^m_H(M),\quad \mathrm{DO}_H^j(M)\subset \Psi^j_H(M),\; j=0,1,2,\ldots
\]
satisfying certain desirable properties. In particular, $\Psi_H(M)$ should contain parametrices for maximally subelliptic differential operators.

This is a well-studied problem and several constructions of an appropriate calculus have been proposed. See for example Melin \cite{Melin-lie-filtrations-1982}, Goodman \cite{Goodman-lnm-1976} and Street \cite{Street-singular-integrals-2014}. We will use the van Erp--Yuncken construction \cite{vanErpYuncken2019} based on the $H$-tangent groupoid. This can be identified with the polyhomogeneous subalgebra of the operators defined by Fermanian-Kammerer, Fischer and Flynn \cite[Section 9.7]{FFF2024arxiv}. We will review the definition and some of the properties of the van Erp--Yuncken calculus in Section \ref{vEY_section} below.

\subsection{Traces on $\Lc_{1,\infty}$}
Here we include some preliminary information about traces. For further details, see \cite{LordSukochevZanin2021}.
Let $\Hc$ be a Hilbert space. Denote by $\Bc(\Hc)$ and $\Kc(\Hc)$ the algebras of bounded and compact linear operators on $\Hc$ respectively. For $T\in \Kc(\Hc),$ the $n$th singular number of $T,$ denoted $\mu(n,T),$ is defined as
\[
    \mu(n,T) = \inf\{\|T-R\|_{\infty}\;:\;\mathrm{rank}(R)\leq n\},\quad n\geq 0.
\]
where $\|\cdot\|_{\infty}$ is the operator norm. Equivalently, $\mu(n,T)$ is the $n$th largest eigenvalue of $|T|,$ starting with $n=0,$ counting multiplicities. The weak trace class $\Lc_{1,\infty}(\Hc),$ for brevity denoted $\Lc_{1,\infty},$ is the ideal of compact operators $T$ such that
\[
    \|T\|_{1,\infty} := \sup_{n\geq 0} (n+1)\mu(n,T)<\infty.
\]
This should be contrasted with the trace class $\Lc_1,$ which is the ideal of $T\in \Kc(\Hc)$ such that
\[
    \|T\|_1 := \sum_{n=0}^\infty \mu(n,T) < \infty.
\]
The trace ideal $\Lc_1$ is the natural domain of the operator trace $\Tr,$ defined on a positive operator $T\in \Lc_1$ by
\[
    \Tr(T) = \sum_{n=0}^\infty \mu(n,T).
\]
The operator trace is additive on positive operators, and extends by linearity to a continuous linear functional on $\Lc_1.$

In general, a trace on an ideal of compact operators is a unitarily invariant funcitonal. The weak trace class admits nontrivial traces, i.e. there are nonvanishing linear functionals $\varphi:\Lc_{1,\infty}\to \Cplx$ satisfying
\[
    \varphi(AB)=\varphi(BA),\quad A\in \Lc_{1,\infty}(\Hc),\; B \in \Bc(\Hc).
\]
A trace $\varphi$ on $\Lc_{1,\infty}$ is called {normalised} if
\[
    \varphi(\mathrm{diag}(\{(n+1)^{-1}\}_{n=0}^\infty)) = 1.
\]
That is, $\varphi(A)=1$ for some (and hence every) positive operator $A$ whose singular value sequence is $\mu(n,A) = (n+1)^{-1}.$

The best known normalised traces on $\Lc_{1,\infty}$ are the {Dixmier traces}, defined in terms of an extended limit. An extended limit $\omega$ is a positive linear functional on $\ell_{\infty}(\Ntrl)$ that coincides with the limit on convergent sequences. The Dixmier trace $\Tr_{\omega}$ corresponding to $\omega$ is defined on a positive operator $T\in \Lc_{1,\infty}$ by
\[
    \Tr_{\omega}(T) = \omega\left(\left\{\frac{1}{\log(N+2)}\sum_{n=0}^N \mu(n,T)\right\}_{N=0}^\infty\right).
\]
The Dixmier trace $\Tr_{\omega}$ is additive on the cone of positive operators and extends by linearity to a trace on $\Lc_{1,\infty}.$

A trace $\varphi$ is called continuous if there exists $C>0$ such that $|\varphi(T)|\leq C\|T\|_{1,\infty}.$ Dixmier traces are continuous, but not all traces on $\Lc_{1,\infty}$ are continuous.

\subsection{The Wodzicki residue and Connes' trace theorem}\label{ctt_subsection}
Let $T$ be a classical pseudodifferential operator of order $-d$ on a compact $d$-dimensional manifold $M.$ The Wodzicki residue $\mathrm{Res}_W(T)$ of $T$ can be defined in several equivalent ways, one of which is the following procedure: first, cover $M$ by charts $\{(U_i,x_i)\}_{i=1}^n$ and let $\sigma_{-d}(T)(x_i,\xi_i)$ be the order $-d$ homogeneous component of the symbol of $T$ in the chart $(U_i,x_i).$ This is a smooth function on the bundle $T^*U_i\setminus \{0\}.$ Multiplying $\sigma_{-d}(T)$
by the Liouville measure $dx_id\xi_i$ on $T^*U_i$ defines a density that is homogeneous of order $0$ with respect to dilations in the fibres of $T^*U_i.$ The cosphere bundle $S^*U_i = (T^*U_i\setminus \{0\})/\Rl_+$ is a sphere bundle over $U_i,$ and the product $\sigma_{-d}(T)(x_i,\xi_i)dx_id\xi_i$ can be proved to descend to a density on $S^*U_i.$ Let $\{\phi_i\}_{i=1}^n$ be a partition of unity subordinate to the cover $\{U_i\}_{i=1}^n.$ We define
\[
    \mathrm{Res}_W(T) := (2\pi)^{-d}\sum_{i=1}^n \int_{S^*U_i}\phi_i(x_i)\sigma_{-d}(T)(x_i,\xi_i)\,dx_id\xi_i.
\]
The Wodzicki residue of $T\in \Psi^{-d}(M)$ depends only on $T$ and is independent of all of the other choices made in its definition. The Wodzicki residue is also called the noncommutative residue.

Wodzicki proved that $\Res_W$ is a trace on the algebra of classical pseudodifferential operators of integer order (i.e., it vanishes on commutators), and moreover it is the essentially unique trace \cite{WodzickiLNM1987}. This suggests an interpretation of $\Res_W(T)$ as a spectral invariant, and this is what Connes' trace theorem achieves.

We consider pseudodifferential operators on $M$ as being operators
\[
    T:C^\infty(M)\to C^\infty(M).
\]
If we select a (nowhere vanishing) density $\nu$ on $M,$ then a pseudodifferential operator $T$ can be identified with a (potentially unbounded) operator on the Hilbert space $L_2(M,\nu).$
It is also possible to define $T$ as an operator on half-densities and therefore to avoid the choice of a density, but we will maintain the convention that pseudodifferential operators act on functions.

If $T$ has order $-d$ on where $d$ is the dimension of $M,$ then $T\in \Lc_{1,\infty}(L_2(M,\nu)).$ Connes' trace theorem computes the value of a normalised trace on $T.$
\begin{theorem}[Connes' trace theorem]\label{ctt_original}
    Let $M$ be a compact $d$-dimensional manifold and let $T\in \Psi^{-d}(M)$ be a classical pseudodifferential operator of order $-d.$ Let $\nu$ be a density on $M.$
    Then $T$ can be identified with an element of $\Lc_{1,\infty}(L_2(M,\nu))$ and for any normalised trace $\varphi$ on $\Lc_{1,\infty}(L_2(M,\nu))$ we have
    \[
        \varphi(T) = \frac{1}{d}\mathrm{Res}_W(T).
    \]
\end{theorem}
Note that while we needed to choose a density $\nu$ in order that $T$ acts on a Hilbert space, a different choice of $\nu$ would result in a unitarily equivalent realisation of $T,$ and the unitary invariance of the trace implies that the value of $\varphi(T)$ is independent of the choice of $\nu.$

No choice of density is needed to define the Wodzicki residue, since it is defined as the canonical integral of a density on $S^*M.$

Connes' trace theorem was first proved by Connes for Dixmier traces on the Banach envelope of $\Lc_{1,\infty}$ \cite[Theorem 1]{Connes-Action}. Connes' proof works verbatim for all continuous normalised traces on $\Lc_{1,\infty},$ and the extension to arbitrary normalised traces was first proved in \cite[Corollary 7.22]{KLPS}. It is important to note that Connes' theorem for continuous traces follows from the spectral asymptotics of negative order pseudodifferential operators proved earlier by Birman and Solomyak, we review this connection in Subsection \ref{weyl_subsection} below.

\subsection{The trace theorem for Carnot manifolds}
We now state our first main result, although some definitions are deferred to later sections. Here, $(M,H)$ is a compact Carnot manifold, and for $m\in \Cplx,$ $\Psi^m_H(M)$ denotes the set of order $m$ pseudodifferential operators on $M$ in the van Erp-Yuncken sense. This will be defined in Section \ref{vEY_section} below.

A residue functional for the van Erp--Yuncken calculus on a filtered manifold was defined by Dave--Haller \cite{DaveHaller2020}, and the same functional was later defined in a different way by Couchet--Yuncken \cite{CouchetYuncken2024}. The residue of $T\in \Psi^{-\dhom}_H(M)$ will be denoted $\mathrm{Res}(T).$ We review both definitions in in Section \ref{residue_section} below.

The following theorem is one of our main results, and is proved in Section \ref{dave_haller_residue}.
\begin{theorem}\label{main_theorem}
    Let $(M,H)$ be a compact Carnot manifold with homogeneous dimension $\dhom,$ and let $T \in \Psi^{-\dhom}_H(M)$ be a pseudodifferential operator on $(M,H)$ in the van Erp--Yuncken sense. Let $\nu$ be a density on $M.$
    Then $T$ can be identified with an element of $\Lc_{1,\infty}(L_2(M,\nu))$ and for any normalised trace $\varphi$ on $\Lc_{1,\infty}(L_2(M,\nu))$ we have
    \[
        \varphi(T) = \frac{1}{\dhom}\mathrm{Res}(T).
    \]
\end{theorem}
Once again, we needed to choose a density $\nu$ so that $T$ could be realised as an operator on a Hilbert space, and we could have avoided this by defining $T$ on the Hilbert space of half-densities on $M.$ Unitary invariance implies that $\varphi(T)$ is independent of $\nu,$ and no choice of density is needed for Couchet--Yuncken's definition of the residue. Theorem \ref{main_theorem} generalises Theorem \ref{ctt_original} and the trace theorem for Heisenberg manifolds due to Ponge \cite[Theorem 4.12]{Ponge2007}.

An equivalent formulation of Theorem \ref{main_theorem} that makes no explicit reference to traces is as follows. Given a compact operator $T,$ let $\{\lambda(n,T)\}_{n=0}^\infty$ denote an enumeration of the eigenvalues of $T$ with algebraic multiplicities and with $|\lambda(n,T)|$ non-increasing. If $T$ has finitely many eigenvalues, then set $\lambda(n,T)=0$ for $n$ greater than the number of nonzero eigenvalues. In particular $\lambda(n,T)=0$ for all $n$ if $T$ is quasinilpotent.
\begin{theorem}\label{main_theorem_e_values_version}
    Let $(M,H)$ be a compact Carnot manifold, and let $T\in \Psi^{-\dhom}_H(M)$ be a pseudodifferential operator on $(M,H)$ in the van Erp-Yuncken sense. Let $\nu$ be a density on $M,$ and identify $T$ with its realisation as an operator on $L_2(M,\nu).$
    As $N\to\infty,$ we have
    \[
        \sum_{n=0}^N \lambda(n,T) = \frac{\mathrm{Res}(T)}{\dhom}\log(N)+O(1).
    \]
\end{theorem}
The equivalence of Theorem \ref{main_theorem} and Theorem \ref{main_theorem_e_values_version} is \cite[Theorem 9.1.2(a)]{LordSukochevZanin2021}.
Just as in Theorem \ref{main_theorem}, the choice of density $\nu$ is irrelevant due to the unitary invariance of the eigenvalue sequence.

\subsection{Relation to Weyl laws and an open problem}\label{weyl_subsection}
In the trivially filtered case, it is possible to make improvements over the eigenvalue asymptotics in Theorem \ref{main_theorem_e_values_version}. It is reasonable to expect that the following is true:
\begin{conjecture}\label{big_conjecture}
    Let $(M,H)$ be a compact Carnot manifold, and let $T\in \Psi^{-\dhom}_H(M)$ be a pseudodifferential operator on $(M,H)$ in the van Erp-Yuncken sense. Let $\nu$ be a density on $M,$ identify $T$ with its realisation as an operator on $L_2(M,\nu),$ and assume that $T$ is self-adjoint with respect to the inner product on $\nu.$
    There exists $\delta>0$ such that as $n\to\infty$ we have
    \[
        \lambda(n,T) = \frac{\mathrm{Res}(T)}{\dhom}n^{-1}+O(n^{-1-\delta}).
    \]
\end{conjecture}
The trivially filtered case has $\delta = \frac{1}{\mathrm{dim}(M)}$ \cite[Theorem 11.8.1]{Ivrii_microlocal_II}.
Obviously Conjecture \ref{big_conjecture} implies Theorem \ref{main_theorem_e_values_version} for self-adjoint $T.$

The form of the asymptotics without an explicit remainder, namely
\begin{equation}\label{only_weyl}
    \lambda(n,T) = \frac{\mathrm{Res}(T)}{\dhom}n^{-1}+o(n^{-1}),\quad n\to \infty
\end{equation}
is also an open problem. The corresponding assertion in the unfiltered case was first proved by Birman and Solomyak \cite{BirmanSolomyakWeaklyPolar1970,BirmanSolomyakAnisotropic1977}. While \eqref{only_weyl} does not imply Theorem \ref{main_theorem_e_values_version}, it does imply that
for all continuous normalised traces $\varphi$ on $\Lc_{1,\infty}$ we have
\[
    \varphi(T) = \frac{1}{\dhom}\mathrm{Res}(T).
\]
In particular, \eqref{only_weyl} is already enough to imply Connes' trace theorem for Dixmier traces.

The original result of Birman--Solomyak was actually stronger than \eqref{only_weyl} in that the eigenvalue sequences of the positive and negative parts of $T,$ $\{\lambda(n,T_+)\}_{n=0}^\infty$ and $\{\lambda(n,T_-)\}_{n=0}^\infty$ were considered separately. We should expect that \eqref{only_weyl} can be strengthened to a statement of the form
\[
    \lambda(n,T_{\pm}) = \frac{\mathrm{Res}(T_{\pm})}{\dhom}n^{-1}+o(n^{-1}),\quad n\to\infty.
\]
The issue with this is that is it not clear what should be meant by $\mathrm{Res}(T_{\pm}).$ In the trivially filtered case, $\mathrm{Res}(T)$ is a functional of the principal symbol $\sigma_{-d}(T)$ which makes sense when $\sigma_{-d}(T)$ is replaced by its positive and negative parts $\sigma_{-d}(T)_{\pm},$ but in the Carnot case we do not yet know how to make sense of $\mathrm{Res}(T_{\pm}).$

The ``soft" proof of Weyl's law for negative order operators due to Ponge \cite{PongeJNCG2023} is likely to be relevant in this setting. The key missing ingredient at this time to emulating Ponge's proof for the Carnot case is Weyl asymptotics for positive order maximally hypoelliptic pseudodifferential operators.

\subsection{The trace theorem for $k$-log-polyhomogeneous operators on Carnot manifolds}
Let $(M,H)$ be a compact Carnot manifold. For $m\in \Cplx$ and $k\geq 0,$ we denote by $\Psi^{m,k}_H(M)$ a space of operators that we call $k$-log-polyhomogeneous operators. These are defined in Subsection \ref{log_poly_subsection} below. In the unfiltered case, $\Psi^{m,k}_H(M)$ reproduces the log-polyhomogeneous calculus of Lesch \cite{Lesch1999}. We will define a $k$th order residue functional $\Res_k$ on $\Psi^{-Q,k}_H(M).$

Operators in $\Psi^{-Q,k}_H(M)$ do not belong to $\Lc_{1,\infty},$ but instead to the ideal $\Ec_{1,k}$ of compact operators $T$ such that
\[
    \mu(n,T) = O(\frac{\log(n+2)^k}{n+1}),\quad n\to\infty.
\]
See Section \ref{ideals_section} for further details. The ideal $\Ec_{1,k}$ has traces, and our version of the trace theorem for $k$-log-polyhomogeneous operators is as follows:
\begin{theorem}\label{k_log_main_thm}
    Let $(M,H)$ be a compact Carnot manifold with homogeneous dimension $\dhom.$ Let $k\geq 1,$ and $T \in \Psi^{-\dhom,k}_H(M).$ For all extended limits $\omega,$ we have
    \[
        \Tr_{\omega,k}(T) = \frac{1}{\dhom(k+1)!}\Res_k(T)
    \]
    where $\Tr_{\omega,k}$ is the Dixmier trace on $\Ec_{1,k}$ corresponding to $\omega.$ Equivalently, we have
    \[
        \sum_{n=0}^N \lambda(n,T) = \frac{1}{\dhom(k+1)!}\Res_k(T)\log(N+2)^{k+1} + o(\log(N+2)^{k+1}),\quad N\to\infty.
    \]
\end{theorem}
Theorem \ref{k_log_main_thm} is proved in Section \ref{log_poly_proof_section}.
We have $\Psi^{m,0}_H(M) = \Psi^m_H(M),$ but note that Theorem \ref{k_log_main_thm} does not recover Theorem \ref{main_theorem_e_values_version} when $k=0.$ The class of traces is smaller (only Dixmier traces as opposed to all traces) and the remainder term is $o(\log(N+2))$ rather than $O(1).$ This is a limitation of the proof technique; the ideals $\Ec_{1,k}$ having a less well-developed theory than $\Lc_{1,\infty}.$ It may be possible to improve the $o(\log(N+2)^{k+1})$ to $O(\log(N+2)^k),$ but this is out of scope.

\subsection{Acknowledgements}
This research was partially financed by the project OpART (ANR-23-CE40-0016) of the Agence Nationale de la Recherche and by the Alexander von Humboldt Stiftung GmbH.

\section{Preliminaries on operator ideals}\label{ideals_section}
Let $\Hc$ be a Hilbert space. Denote by $\Kc(\Hc),$ or simply $\Kc$ for brevity, the ideal of compact operators on $\Hc.$ For $T \in \Kc,$ the singular value function $\mu(T)$ of $T$ is defined by
\[
    \mu(t,T) = \inf\{\|T-R\|_{\infty}\;:\;\mathrm{rank}(R)\leq t\},\quad t \geq 0.
\]
Here, $\|\cdot\|_{\infty}$ is the operator norm.
Note that $\mu(T)$ is a step function, and encodes essentially the same information as the singular value sequence $\{\mu(n,T)\}_{n=0}^\infty.$ We could have also defined $\mu(t,T)$ as the $\lfloor t\rfloor+1$th largest eigenvalue of $|T|,$ ordered with multiplicities. Writing $\mu(T)$ as a function of a positive real argument serves to simplify some computations.

For $0<p<\infty,$ the Schatten ideal $\Lc_p$ is defined by the (quasi)norm
\[
    \|T\|_{\Lc_p} := \left(\int_0^{\infty} \mu(t,T)^{p}\,dt\right)^{1/p}.
\]
While the weak Schatten ideal $\Lc_{p,\infty}$ has quasinorm
\[
    \|T\|_{\Lc_{p,\infty}} := \sup_{t>0}\,t^{\frac1p}\mu(t,T).
\]

The following is a well-known identity for the $K$-functional of the pair $(\Lc_2,\Kc).$ It follows from the formula for the $K$-functional for commutative $L_p$ spaces \cite[Theorem 5.2.1]{Bergh-Lofstrom-1976} and the relationship between $K$-functionals of commutative and noncommutative $L_p$ spaces \cite[Proposition 7.4.2]{DdPS-vapour}.
\begin{lemma}\label{K_function_identity}
    Let $t>0.$ For all $T\in \Kc,$ we have
    \[
        \left(\int_0^{t^2} \mu(s,T)^2\,ds\right)^{1/2} \approx \inf\{\|S\|_{\Lc_2}+t\|T-S\|_{\infty}\;:\; S \in \Lc_2\}.
    \]
    The implied constants in the notation $\approx$ are universal.
\end{lemma}

Generalising the weak Schatten ideals $\Lc_{p,\infty}$ is a class of ideals that we denote $\Ec_{p,k},$ as follows.
\begin{definition}
    Let $0<p<\infty$ and $k\geq 0.$ Say that $T\in \Ec_{p,k}$ if
    \[
        \|T\|_{\Ec_{p,k}} := \sup_{t>0}\,t^{\frac{1}{p}}\log(t+1)^{-k} \mu(t,T) < \infty.
    \]
\end{definition}
Note that $\Ec_{p,0} = \Lc_{p,\infty}.$ The singular value inequalities
\[
    \mu(t+s,T+S)\leq \mu(t,T)+\mu(s,S),\,\mu(t+s,TS)\leq \mu(t,T)\mu(s,S).
\]
readily imply that $\Ec_{p,k}$ is a ideal, and that
\begin{equation}\label{holder_type}
    \Ec_{p,k}\cdot \Ec_{q,\ell} \subseteq \Ec_{r,k+\ell},\quad p,q,k,\ell>0
\end{equation}
where
\[
    \frac{1}{r}=\frac{1}{p}+\frac{1}{q}.
\]

The following lemma relates the quantity in Lemma \ref{K_function_identity} with the $\Ec_{p,k}$-quasinorm.
\begin{lemma}\label{Hardy_type_inequality}
    Let $2<p<\infty$ and $k\geq 0.$ For $T \in \Kc,$ we have
    \[
        \|T\|_{\Ec_{p,k}} \approx \sup_{t>0} \,t^{\frac{2}{p}-1}\log(t+1)^{-k}\left(\int_0^{t^2} \mu(s,T)^2\,ds\right)^{\frac12}.
    \]
    The implied constants depend only on $k$ and $p.$
\end{lemma}
\begin{proof}
    Since $\mu(T)$ is monotone decreasing, we have
    \[
        \mu(t^2,T) \leq \frac{1}{t}\left(\int_0^{t^2}\mu(s,T)^2\,ds\right)^{1/2}
    \]
    Therefore
    \[
        \sup_{t>0} t^{\frac{2}{p}}\log(t^2+1)^{-k}\mu(t^2,T) \leq \sup_{t>0} \,t^{\frac{2}{p}-1}\log(t^2+1)^{-k}\left(\int_0^{t^2} \mu(s,T)^2\,ds\right)^{\frac12}.
    \]
    Hence
    \[
        \|T\|_{\Ec_{p,k}} \lesssim \sup_{t>0} \,(t+1)^{\frac{2}{p}-1}\log(t+1)^{-k}\left(\int_0^{t^2} \mu(s,T)^2\,ds\right)^{\frac12}.
    \]
    On the other hand, using
    \[
        \mu(s,T) \leq \|T\|_{\Ec_{p,k}}s^{-\frac{1}{p}}\log(s+1)^k
    \]
    we have
    \[
        \left(\int_0^{t^2}\mu(s,T)^2\,ds\right) \leq \|T\|_{\Ec_{p,k}}\left(\int_0^{t^2} s^{-\frac{2}{p}}\log(s+1)^{2k}\,ds\right)^{1/2}
    \]
    Since $p>2,$ this latter integral converges and
    \[
    \left(\int_0^{t^2} s^{-\frac{2}{p}}\log(s+1)^{2k}\,ds\right)^{1/2} \approx t^{1-\frac{2}{p}}\log(t+1)^k.
    \]
    Therefore
    \[
        \sup_{t>0} t^{\frac{2}{p}-1}\log(t+1)^{-k}\left(\int_0^{t^2}\mu(s,T)^2\,ds\right)^{1/2} \lesssim \|T\|_{\Ec_{p,k}}.
    \]

\end{proof}

\subsection{Traces on $\Ec_{1,k}$}
A trace on an ideal $\Ic\subseteq \Bc(\Hc)$ is a unitarily invariant linear functional. That is,
\[
    \varphi:\Ic\to \Cplx,\quad \varphi(UAU^*) = \varphi(A)
\]
for all $A \in \Ic$ and unitary operators $U.$ Equivalently, $\varphi(AB)=\varphi(BA)$ for all $A\in \Ic$ and all $B\in \Bc(\Hc).$

There exist continuous traces on $\Ec_{1,k}$ for all $k\geq 0.$ These can be constructed in a similar way to Dixmier traces on $\Lc_{1,\infty}$.

An extended limit $\omega$ is by definition a state on the $C^*$-algebra $\ell_{\infty}(\Ntrl)$ which vanishes on finitely supported sequences. That is, $\omega$ is a positive continuous linear functional such that $\omega(1)=1$ and $\omega(x)=0$ when $\lim_{n\to\infty} x_n=0.$ For $0\leq T \in \Ec_{1,k},$ we define
\[
    \Tr_{\omega,k}(T) := \omega\left(\left\{\frac{1}{\log(N+2)^{k+1}}\sum_{n=0}^N \mu(n,T)\right\}_{N=0}^\infty\right)
\]
\begin{lemma}
    For all $k\geq 0$ and $0\leq T,S\in \Ec_{1,k},$ we have
    \[
        \Tr_{\omega,k}(T+S) = \Tr_{\omega,k}(T)+\Tr_{\omega,k}(S),\quad 0\leq T,S\in \Ec_{1,k}.
    \]
\end{lemma}
\begin{proof}
    For all $N\geq 0,$ we have
    \[
        \sum_{n=0}^N \mu(n,T+S) \leq \sum_{n=0}^N \mu(n,T)+\mu(n,S)\leq \sum_{n=0}^{2N} \mu(n,T+S).
    \]
    Both of these inequalities follow from the variational characterisation of the sum of the first $N$ singular values, See e.g. \cite[Lemma 4.1]{GohbergKrein}.

    Dividing by $\log(N+2)^{k+1},$ we have
    \[
        0\leq \frac{1}{\log(N+2)^{k+1}}\sum_{n=0}^N \mu(n,T+S)-\mu(n,T)-\mu(n,S) \leq \|T+S\|_{\Ec_{1,k}}\cdot \frac{1}{\log(N+2)^{k+1}}\sum_{n=N+1}^{2N} \frac{\log(n+1)^k}{n}.
    \]
    This is $o(1)$ as $N\to\infty.$ Applying $\omega,$ we conclude that
    \[
        \Tr_{\omega,k}(T+S)=\Tr_{\omega,k}(T)+\Tr_{\omega,k}(S).
    \]
\end{proof}
Hence $\Tr_{\omega,k}$ extends by linearity to a trace on $\Ec_{1,k},$ which we call the \emph{Dixmier trace} on $\Ec_{p,k}$ corresponding to $\omega.$
By definition, if $T\in \Ec_{1,k},$ then we have
\begin{align*}
    \Tr_{\omega,k}(T) &= \omega(\{\frac{1}{\log(N+2)^{k+1}}\sum_{n=0}^N \mu(n,(T+T^*)_+)-\mu(n,(T+T^*)_-)\\
    &\quad+i\mu(n,(T-T^*)_+)-i\mu(n,(T-T^*)_-)\}_{N=0}^{\infty}).
\end{align*}
where $X_{\pm}$ are the positive and negative parts of a self-adjoint operator $X.$ It is not obvious, but nonetheless true, that $\Tr_{\omega,k}(T)$ can be computed from the eigenvalues of $T$ alone. The important feature of $\Ec_{1,k}$ here is that it is geometrically stable, i.e. that if $S\in \Ec_{1,k}$ and
\[
    \mu(n,T) \leq \left(\prod_{j\leq n} \mu(j,S)\right)^{\frac{1}{n+1}}
\]
then $T \in \Ec_{1,k}.$ Indeed, since $\mu(j,S)\lesssim \frac{\log(j+2)^k}{j},$ we have
\[
    \mu(n,T) \lesssim \log(n+2)^k\frac{1}{((n+1)!)^{\frac{1}{n+1}}}\lesssim \frac{\log(n+2)^k}{n+1}.
\]

For a compact operator $T,$ let $\{\lambda(n,T)\}_{n=0}^{\infty}$ denote an enumeration of the eigenvalues of $T$ in order of non-increasing absolute value. If $T$ has finitely many eigenvalues, put $\lambda(n,T)=0$ for $n$ larger than the number of nonzero eigenvalues.

Note that there is a certain ambiguity in the choice of enumeration of $\lambda(n,T),$ since we do not specify how to choose between eigenvalues of the same magnitude. This ambiguity will not end up being important. Note that however we order the eigenvalues, when $T$ is self-adjoint we have $|\lambda(n,T)| = \lambda(n,|T|)=\mu(n,T).$

\begin{theorem}\label{spectrality}
    Let $k\geq 0,$ and let $\omega$ be an extended limit. For all $T \in \Ec_{1,k},$ we have
    \[
        \Tr_{\omega,k}(T) = \omega\left(\left\{\frac{1}{\log(N+2)^{k+1}}\sum_{n=0}^N \lambda(n,T)\right\}_{N=0}^\infty\right)
    \]
\end{theorem}
\begin{proof}
    Initially suppose that $T$ is self-adjoint. Then by definition we have
    \[
        \Tr_{\omega,k}(T) = \omega\left(\{\frac{1}{\log(N+2)^k}\sum_{n=0}^N \mu(n,T_+)-\mu(n,T_-)\}_{N=0}^\infty\right).
    \]
    The eigenvalue sequence of $T$ is the sequence $\{\mu(n,T_+)\}_{n=0}^\infty\sqcup \{-\mu(n,T_-)\}_{n=0}^\infty,$ arranged in decreasing absolute value. Therefore,
    \[
        \left|\sum_{n=0}^N \mu(n,T_+)-\mu(n,T_-)-\lambda(n,T)\right| \leq \sum_{n\leq N,\,\mu(n,T_+)\leq |\lambda(N,T)|} \mu(n,T_+) + \sum_{n\leq N,\, \mu(n,T_-)\leq |\lambda(N,T)|} \mu(n,T_-)
    \]
    Weyl's inequality (see e.g. \cite[Lemma 3.3]{GohbergKrein}) asserts that
    \[
        \prod_{n=0}^N |\lambda(n,T)| \leq \prod_{n=0}^N \mu(n,T).
    \]
    The geometric stability of $\Ec_{1,k}$ mentioned above implies that there exists a constant $C_T$ such that
    \[
        |\lambda(N,T)| \leq C_T \frac{\log(N+2)^k}{N+1}.
    \]
    Therefore
    \[
        \sum_{n\leq N,\,\mu(n,T_+)\leq |\lambda(N,T)|}\mu(n,T_+) \leq \sum_{n=0}^N C_T\frac{\log(N+2)^k}{N} \leq 2C_T\log(N+2)^k.
    \]
    Similarly
    \[
        \sum_{n\leq N,\,\mu(n,T_-)\leq |\lambda(N,T)|}\mu(n,T_-) \leq \sum_{n=0}^N C_T\frac{\log(N+2)^k}{N} \leq 2C_T\log(N+2)^k.
    \]
    So, for self-adjoint $T$ we have
    \[
        \sum_{n=0}^N \lambda(n,T) = \sum_{n=0}^N \mu(n,T_+)-\mu(n,T_-) + O(\log(N+2)^k).
    \]
    This concludes the proof in the case that $T$ is self-adjoint.
%
    Now consider the case that $T$ is normal.
    By Lemma 5.2.9 of \cite{LordSukochevZanin2021}, for a normal operator $T$ we have
    \[
        \left|\sum_{n=0}^N \lambda(n,T)-\lambda(n,\Re(T))-i\lambda(n,\Im(T))\right| \leq 5(N+1)\mu(N,T),\quad N\geq 0.
    \]
    In this case, $T\in \Ec_{1,k},$ so $(N+1)\mu(N,T)\lesssim \log(N+2)^{k}.$ Hence,
    \[
        \lim_{N\to\infty} \frac{1}{\log(N+2)^{k+1}}\left|\sum_{n=0}^N \lambda(n,T)-\lambda(n,\Re(T))-i\lambda(n,\Im(T))\right| = 0.
    \]
    Therefore for any extended limit we have
    \[
        \omega\left(\left\{\frac{1}{\log(N+2)^k}\sum_{n=0}^N \lambda(n,T)-\lambda(n,\Re(T))-i\lambda(n,\Im(T))\right\}\right) = 0.
    \]
    Since $\Re(T)$ and $\Im(T)$ are self-adjoint, we conclude that
    \[
        \Tr_{\omega,k}(T) = \Tr_{\omega,k}(\Re(T))+\Tr_{\omega,k}(\Im(T)) =  \omega\left(\left\{\frac{1}{\log(N+2)^{k+1}}\sum_{n=0}^N \lambda(n,T)\right\}_{N=0}^\infty\right).
    \]
    So, the result is true when $T$ is normal.

    Finally, let $\lambda(T)$ denote the diagonal operator
    \[
        \lambda(T) = \mathrm{diag}\{\lambda(n,T)\}_{n=0}^\infty.
    \]
    Theorem 8 of \cite{SZ-AiM} says that since $\Ec_{1,k}$ is geometrically stable, for any trace $\varphi$ on $\Ec_{1,k}$ we have $\varphi(T)=\varphi(\lambda(T)).$ Since $\lambda(T)$ is normal, this completes the proof.
\end{proof}

\section{Preliminaries on the van Erp-Yuncken calculus}\label{vEY_section}
    As mentioned in the introduction, we will adopt the following notation for sections: given a vector bundle $E\to M,$ $C^\infty(M,E)$ denotes the space of smooth sections of $E,$ while $C^0(M,E)$ denotes the space of continuous sections. Both $C^\infty(M,E)$ and $C^0(M,E)$ are given their standard Fr\'echet topologies, and Banach space topology in the case of $C^0(M,E)$ when $M$ is compact.
\subsection{Densities and fibred distributions}
    For a vector bundle $E\to M,$ denote $\dens E\to M$ the corresponding density bundle i.e., the line bundle over $M$ whose transition mappings are the absolute values of the inverses of the determinants of the transition mappings
    of $E.$ We abbreviate $\dens M := \dens TM.$ For a submersion of manifolds
    \[
        \pi:M\to N
    \]
    the bundle $\dens_\pi\to M$ of densities transversal to the $r$-fibres is defined as
    \[
        \dens_\pi := \dens \ker(d\pi).
    \]
    For details, see \cite{LescureManchonVassout2017}.
    We will work in the same framework as \cite{vanErpYuncken2019}, and we will follow their convention about distributions: define $\Ec'(M),$ the space of compactly supported distributions on $M,$ as the continuous linear dual of $C^\infty(M).$ A compactly supported continuous density in $C^0(M,\dens M)$ can be identified with a distribution by integration. Distributions in this convention should be thought of as generalised densities, not generalised functions. We will denote the pairing of distributions with smooth functions by angle brackets $\langle \cdot,\cdot\rangle.$ We will distinguish notationally between distributions and functions by using blackboard bold letters ($\Pb,\kbb$ etc.) for distributions and unbolded letters for functions.

    Given a smooth map $f:M\to N,$ the pullback $f^*:C^\infty(N)\to C^\infty(M)$ are defined as usual on functions by
    \[
        f^*u = u\circ f,\quad u\in C^\infty_c(M).
    \]
    The pushforward $f_*$ is defined on distributions by duality, namely
    \[
        \langle f_*\Tb,u\rangle  = \langle T,f^*u\rangle,\quad \Tb \in \Ec'(M),\; u\in C^\infty(N).
    \]
    \begin{definition}
        If $\pi:M\to N$ is a submersion, the space of $\pi$-fibred distributions $\Ec'_\pi(M)$ is the space of continuous $C^\infty(N)$-linear functionals
        \[
            \Tb:C^\infty(M)\to C^\infty(N)
        \]
        where $C^\infty(M)$ is a $C^\infty(N)$-module by $(f\cdot g)(m) = f(\pi(m))g(m)$ where $g \in C^\infty(M)$ and $f \in C^\infty(N).$ The space $\Ec'_\pi(M)$ is equipped with the topology of bounded convergence \cite[p. 337]{Treves1967}. This is defined by the neighbourhood basis of zero given by
        \[
            U_{B,V} = \{\Tb\in \Ec'_r(M)\;:\; \Tb(B)\subset V\}
        \]
        where $B$ ranges over bounded subsets of $C^\infty(M)$ and $V$ ranges over neighbourhoods of zero in $C^\infty(N).$

        We denote the $C^\infty(N)$-valued pairing of $\pi$-fibred distributions with smooth functions by round brackets $(\cdot,\cdot).$
    \end{definition}
    Every $\pi$-properly supported $\pi$-fibred density $\Tb \in C^\infty(M,\dens_\pi)$ defines an $\pi$-fibred distribution $\Tb \in \Ec'_\pi(M)$ by integration. This essentially coincides with the notion of $\pi$-transversal distributions in \cite[Definition 2.3]{LescureManchonVassout2017}, \cite[Section 1.2.1]{AndroulidakisSkandalis2011} with two caveats: first, that in those papers distributions were defined as linear functionals on densities rather than functions, and second that we do not make any restrictions on the support in $N.$
    Related to this we mention two separate reasons why the na\"ive inclusion $\Ec'_\pi(M)\subseteq \Ec'(M)$ is false. The first reason is that to identify an $\pi$-fibred distribution with a distribution on $M,$ we must select a density on the base space $N.$ The second reason is that we do not assume any compact support condition on $\pi$-fibred distributions in the base space $N.$

    Since every submersion is locally a coordinate projection, the notion of $\pi$-fibred distribution is essentially contained within the notion of semiregular distribution in the sense of Tr\`eves \cite[p.532]{Treves1967}.

    The most relevant feature of the topology on $\Ec'_\pi(M)$ is that with this topology $\Ec'_\pi(M)$ is a quasi-complete locally convex space. To see this, we can give a more explicit description of the topology as follows. Let $\{p^{M}_n\}_{n\geq 0}$ and $\{p^N_n\}_{n\geq 0}$ be sequences of seminorms for $C^\infty(M)$ and $C^\infty(N)$ respectively. We can assume without loss of generality that
    \[
        p_0^M \leq p_1^M \leq p_2^M\leq \cdots,\; \quad p_0^N \leq p_1^N\leq p_2^N\leq \cdots.
    \]
    Every bounded subset of $C^\infty(M)$ is contained inside a set of the form
    \[
        B_x = \{u\;:\; p_n^M(u)\leq x_n,\quad n\geq 0\}
    \]
    where $x_0\leq x_1\leq \cdots$ is an increasing sequence. By the Arzela-Ascoli theorem, $B_x$ is compact. Since every closed and bounded subset is compact, $C^\infty(M)$ is a so-called Fr\'echet-Montel space \cite[p.112]{Schwartz1973}. The space $\Lc_b(C^\infty(M),C^\infty(N))$ is a locally convex space when equipped with the seminorms $\{p_{n,x}\}$ parametrised by $n\geq 0$ and increasing sequences $x$ given by
    \[
        p_{n,x}(T) = \sup\{p_n^N(Tu)\;:\; u\in B_x\}.
    \]
    We have defined $\Ec'_\pi(M)$ as a subset
    \[
        \Ec'_\pi(M)\subset \Lc_b(C^\infty(M),C^\infty(N))
    \]
    and equipped it with the subspace topology.
    It is straightforward to check that $\Ec'_\pi(M)$ is closed with the bounded convergence topology. Since $C^\infty(N)$ is a Fr\'echet-Montel space and $\Ec'_\pi(M)$ is a closed subspace of $\Lc_b(C^\infty(M),C^\infty(N)),$ it follows that $\Ec'_{\pi}(M)$ is a locally convex Suslin space in the sense of Thomas \cite{ThomasTAMS75}. The main results of \cite{ThomasTAMS75} imply that there is a good integration theory for functions valued in $\Ec'_{\pi}(M).$ We will only use the fact that if $f:[a,b]\to \Ec'_{\pi}(M)$ is continuous and bounded, then the integral $\int_{a}^b f(t)\,dt$ exists and is differentiable as an $\Ec'_{\pi}(M)$-valued function of $b$ \cite[Theorem 11]{ThomasTAMS75}, with
    \[
        \frac{d}{db}\int_{a}^b f(t)\,dt = f(b).
    \]

    Given a linear map $T:C^\infty(N)\to C^\infty(M),$ its distributional kernel in the sense of Schwartz is denoted by the same letter $T.$ We write
    \[
        Tu(x) = \langle T(x,\cdot),u\rangle,\quad x  \in M,\; u\in C^\infty(N).
    \]
    We will freely identify operators with their Schwartz kernels. Observe that the kernel of a linear map $T:C^\infty(N)\to C^\infty(M)$ is an element of $\Ec'_\pi(M\times N),$ where $\pi:M\times N\to M$ is the coordinate projection.
    We say that an operator $T$ is smoothing if $T$ (identified with its kernel) belongs to $C^\infty(M\times N,\dens_\pi).$

\subsection{The $H$-tangent groupoid}\label{subsec_groupoid_def}
    Recall that a Lie groupoid consists of the following data: a pair of smooth manifolds $(G,G^{(0)}),$ a pair of submersions
    \[
        r,s:G\to G^{(0)}
    \]
    a smooth map $^{-1}:G\to G$ called the inverse, and a smooth map
    \[
        \circ:G^{(2)}\to G
    \]
    called the composition,
    where $G^{(2)}\subseteq G^2$ is the submanifold of composable pairs
    \[
        G^{(2)} = \{(g,h)\in G^2\;:\; r(h)=s(g)\}.
    \]
    The composition and inverse mapping satisfy the condition that $G$ is a small category with inverses and object space $G^{(0)}.$

    We will consider only three Lie groupoids in this paper: the pair groupoid $G=M\times M$ of a smooth manifold $M,$ the osculating groupoid $G = T_HM$ and the $H$-tangent groupoid $G=\Tb_HM$ of a Carnot manifold $(M,H).$ The object space $G^{(0)}=M$ is the same for $G=M\times M$ and $G=T_HM.$ The object space of $\Tb_HM$ is $M\times \Rl.$ We will use the same symbols $r,s$ to denote the range and source submersions in all three groupoids, the precise groupoid in any given case should be clear from context.

    The pair groupoid has range and source maps
    \[
        r(x,y) = x,\; s(x,y) = y
    \]
    composition and inverse
    \[
        (x,y)\circ (y,z) = (x,z),\; (x,y)^{-1} = (y,x).
    \]
    Let $(M,H)$ be a Carnot manifold, as described in Subsection \ref{carnot_intro_section}. We will give a brief description of the $H$-tangent groupoid as defined by van Erp--Yuncken and also by Choi--Ponge. For more details see \cite{vanErpYuncken2017,vanErpYuncken2019,ChoiPonge2018}.
    The osculating groupoid $T_HM$ has already been defined in Section \ref{carnot_intro_section}. The range and source maps are the bundle maps
    \[
        r,s:T_HM\to M.
    \]
    The composition and inverse maps are defined as the fibrewise Lie group multiplication and inversion.

    The $H$-tangent groupoid $\Tb_HM$ takes more effort to define.
    Recall that we have fixed a splitting $\psi:\tf_HM\to TM,$ i.e. an isomorphism with the property that $\psi|_{\tf_HM^k}$ is right-inverse to the quotient $H_k\to H_k/H_{k-1}$ for $k=1,\ldots,N.$ We choose a connection $\nabla^{\psi}$ on $\tf_HM$ with the property that
    \[
        \nabla^{\psi}_X(\delta_t Z) = \delta_t \nabla^{\psi}_XZ,\quad t > 0,\; X \in C^\infty(M, TM),\; Z \in C^\infty(M, \tf_HM).
    \]
    Let $\exp$ denote the exponential with respect to this connection. That is, for each $x \in M,$ there is a local diffeomorphism
    \[
        \exp_x:\tf_HM_x\to M
    \]
    such that for all sufficiently small $v \in \tf_HM_x$ and $\varepsilon>0,$ the path
    \[
        \{\exp_x^{\psi}(tv)\}_{-\varepsilon< t < \varepsilon}
    \]
    is parallel with respect to $\nabla^{\psi}.$

    The bundle $\tf_HM$ is locally trivial as a bundle of graded vector spaces. This means that near any point of $M,$ we can find an open subset $U$ and a local trivialisation of $\tf_HM$
    \[
        \tf_HU \approx U\times V
    \]
    where $V=\Rl^d$ is a fixed graded vector space equipped with dilations $\{\delta_t\}_{t>0}$ and homogeneous dimension $\dhom.$

    The $H$-tangent groupoid is defined as the disjoint union
    \[
        \Tb_HM := T_HM\times \{0\}\sqcup M\times M\times \Rl^{\times}
    \]
    which is a manifold when equipped with the smooth structure on $M\times M\times \Rl^{\times}$ and by declaring the functions
    \[
        \mathrm{Exp}^{\psi}:U\times V\times (-\varepsilon,\varepsilon)\to M\times M\times \Rl,\quad \mathrm{Exp}(x,z,h) = (x,\exp_x^{\psi}(\delta_hz),h)
    \]
    to be smooth for any local trivialisation $U\times V \approx \tf_HU,$ and sufficiently small $\varepsilon.$
    The set $U\times V\times (-\varepsilon,\varepsilon)$ is called a local exponential coordinate chart for $\Tb_HM.$
    A \emph{global} exponential coordinate chart $\Gexp^{\psi}$ is an extension of $\mathrm{Exp}^{\psi}$ to $\tf_HM\times \Rl,$ written as
    \[
        \Gexp^{\psi}:\tf_HM\times \Rl\to \Tb_HM
    \]
    (although $\Gexp$ is defined only on a neighbourhood of $\tf_HM\times \{0\}$ in $\tf_HM\times \Rl$) and given by
    \[
        \Gexp^{\psi}((x,z),h) = \begin{cases}
                                    ((x,z),0),\quad h=0,\\
                                    (x,\exp_x^{\psi}(\delta_hz),h),\quad h\neq 0.
                                \end{cases}
    \]
    for $(x,z)\in \tf_HM$ and sufficiently small $h.$

    The range and source maps on for the groupoid $\Tb_HM,$
    \[
        s,r:\Tb_HM\to M\times \Rl
    \]
    and are defined as
    \begin{align*}
        r(x,y,h) &= x,\quad h\neq 0,\; r((x,z),0) = x,\; z \in \tf_HM_x,\\
        s(x,y,h) &= y,\quad h\neq 0,\; s((x,z),0) = x,\; z\in \tf_HM_x.
    \end{align*}
    These define submersions of $\Tb_HM.$
    The groupoid operation is given by
    \begin{align*}
        (x,y,h)\circ (y,w,h) &= (x,w,h),\quad x,y,w\in M,\; h>0,\\\
        (x,z_1)\circ (x,z_2) &= (x,z_1z_2),\; x\in M,\; z_1,z_2\in T_HM_x.
    \end{align*}
    where $z_1z_2$ is the group operation in $T_HM_x.$
    These choices give $\Tb_HM$ the structure of a Lie groupoid \cite{vanErpYuncken2017}.
    Note that distributions $\Pb\in \Ec'_r(\Tb_HM)$ are $r$-proper in the sense that $r:\mathrm{supp}(\Pb)\to M\times \Rl$ is proper.
    Following \cite{vanErpYuncken2019}, we define the space $C^\infty_p(\Tb_HM,\dens_r)$ of $r$-fibred densities $\fbb$ whose support is $r$ and $s$ proper in the sense that the maps $r,s:\mathrm{supp}(\fbb)\to M\times \Rl$ are proper.

    We will use throughout this paper the notation $\Pb_h$ for the value of an $r$-fibred distribution at $h \in \Rl.$ Since this notation is important we emphasise it here.
    \begin{definition}\label{restriction}
        Let $\Pb \in \Ec'_r(\Tb_HM).$ For $h\in \Rl,$ let $\Pb_h$ be the restriction of $\Pb$ to $r^{-1}(M\times \{h\}).$ To be precise, for $h\neq 0$ we have
        \[
            \Pb_h\in \Ec'_r(M\times M)
        \]
        and
        \[
            \Pb_0 \in \Ec'_r(T_HM).
        \]
    \end{definition}

\subsection{The zoom action and the zoom vector field}
    The zoom action $\alpha$ is the group of automorphisms of $\Tb_HM$ defined by
    \[
        \alpha_\lambda(x,y,h) = (x,y,\lambda^{-1}h),\quad \alpha_\lambda((x,z),0) = ((x,\delta_\lambda z),0),\quad \lambda \in \Rl_+.
    \]
    Observe that for any $f \in C^\infty(\Tb_HM),$ the pullback map
    \[
        \lambda\mapsto \alpha_{\lambda}^*f = f\circ\alpha_{\lambda}
    \]
    is smooth as a map from $\Rl_+$ to the Fr\'echet space $C^\infty(\Tb_HM).$
    The pushforward by $\alpha_{\lambda}$ is defined on distributions $\Pb \in \Ec'(\Tb_HM)$ by
    \[
        \langle (\alpha_{\lambda})_*\Pb,f\rangle = \langle \Pb,\alpha_{\lambda}^*f\rangle,\quad f \in C^\infty(\Tb_HM)
    \]
    where recall from above that $\langle \cdot,\cdot\rangle$ is the $\Cplx$-valued pairing of distributions on $\Tb_HM$ with smooth functions on $\Tb_HM.$ For an $r$-fibred distribution $\Pb \in \Ec'_r(\Tb_HM)$, we have
    \[
        ((\alpha_{\lambda})_*\Pb,f)(x,h) = (\Pb,\alpha_{\lambda}^*f)(x,\lambda h),\quad f\in C^\infty(\Tb_HM),\, x\in M\; h\in \Rl, \lambda>0.
    \]
    Here, again recall from above that $(\cdot,\cdot)$ is the $C^\infty(M\times\Rl)$-valued pairing of $r$-fibred distributions on $\Tb_HM$ with smooth functions on $\Tb_HM.$
    Note also that
    \[
        ((\alpha_{\lambda})_*\Pb)_h = \Pb_{\lambda h},\quad \lambda>0,h\in \Rl
    \]
    where the restriction to $h\in \Rl$ is as in Definition \ref{restriction}.

\begin{lemma}\label{smoothness_and_boundedness}
    Let $\Pb \in \Ec'_r(\Tb_HM).$
    \begin{enumerate}[{\rm (i)}]
        \item{}\label{smoothness} The pushforward
                \[
                    \lambda\mapsto (\alpha_{\lambda})_*\Pb
                \]
                is smooth as a map from $\Rl_+$ to $\Ec'_r(\Tb_HM)$ with the bounded convergence topology.
        \item{}\label{boundedness} The family
                \[
                    \{(\alpha_{\lambda})_*\Pb\}_{0<\lambda\leq 1}
                \]
                is bounded in $\Ec'_r(\Tb_HM).$
        \item{}\label{boundedness_2}
            Let $\fbb \in C^\infty_p(\Tb_HM,\dens_r).$ The family
            \[
                \{\lambda^{\dhom}(\alpha_{\lambda})_*\fbb\}_{0<\lambda\leq 1}
            \]
            is bounded in $C^0(\Tb_HM,\dens_r).$
    \end{enumerate}
\end{lemma}
\begin{proof}
    First we note that if $f \in C^\infty(\Tb_HM),$ then the function
    \[
        \lambda\to \alpha_{\lambda}^*f
    \]
    is smooth as a map from $\Rl_+$ to $C^\infty(\Tb_HM).$ This is immediate from the smoothness of $f$ and of $\alpha.$

    In local exponential coordinates $U\times V\times \Rl$ near a point $(x,h)\in M\times \Rl,$ we have
    \[
        ((\alpha_{\lambda})_*\Pb,f)(x,h) = \int_{\tf_HM_x} P(x,z,\lambda h)f(x,\delta_{\lambda}z,h)\,dz.
    \]
    where $P$ is a smooth function in $(x,h)$ and a compactly supported distribution in $z.$ Since $P(x,z,h)$ is supported in a compact set of $z \in V$ and $\lambda<1,$ it follows that only the values of $z\mapsto f(x,z,h)$ in a compact set of $z\in V$ are relevant. The smoothness \eqref{smoothness} and boundedness \eqref{boundedness} in these coordinates are immediate.

    Finally, we check \eqref{boundedness_2}. In local exponential coordinates, $\lambda^{-Q}(\alpha_{\lambda})_*\fbb$ is given by
    \[
        \lambda^{Q}(\alpha_{\lambda})_*\fbb(x,z,h) = f(x,\delta_{\lambda}^{-1}z,\lambda h)|dz|
    \]
    for some smooth function $f,$ compactly supported in the second variable. For any compact set $K\subset U\times V\times \Rl,$ it is clear that $(x,z,h)\mapsto f(x,\delta_{\lambda}^{-1}z,\lambda h)$ is bounded in $C^0(\Tb_HM)$ as $\lambda\to 0,$ and hence that $\lambda^Q(\alpha_{\lambda})_*\fbb$ is bounded in $C^0(\Tb_HM,\dens_r).$

\end{proof}

The vector field generating the zoom action will be important for defining the calculus.
\begin{definition}
    Denote by $Z$ the (negative) generator of the semigroup $\{\alpha_\lambda\}_{\lambda>0}$, which acts on functions by
    \[
        Zf := -\frac{d}{d\lambda}|_{\lambda=1}f\circ \alpha_\lambda,\quad f \in C^\infty(\Tb_HM).
    \]
\end{definition}
The Lie derivative by $Z$ acts on distributions by
\[
    \Lc_Z\Pb := \frac{d}{d\lambda}|_{\lambda=1}(\alpha_\lambda)_*\Pb,\quad \Pb\in \Ec'(\Tb_HM).
\]
Equivalently,
\[
    \langle \Lc_Z\Pb,f\rangle = \langle \Pb,-Zf\rangle,\quad f \in C^\infty(\Tb_HM).
\]
    We use the same notation for the derivative by $Z$ on $r$-fibred distributions, that is
    \[
        \Lc_Z:\Ec'_r(\Tb_HM)\to \Ec'_r(\Tb_HM)
    \]
    is defined by
    \[
        \Lc_Z\Pb := \frac{d}{d\lambda}|_{\lambda=1}(\alpha_\lambda)_*\Pb,\quad \Pb\in \Ec_r'(\Tb_HM).
    \]
    By Lemma \ref{smoothness_and_boundedness}, this derivative makes sense in the bounded convergence topology of $\Ec'_r(\Tb_HM).$
    Since $\alpha$ is a semigroup of automorphisms of the groupoid $\Tb_HM,$ $\Lc_Z$ is a derivation on the corresponding convolution algebra.

    We will denote by the same letter the vector field $Z$ on $T_HM$ that generates the canonical dilations, that is
    \[
        Zf = -\frac{d}{d\lambda}|_{\lambda=1}f\circ \delta_{\lambda},\quad f \in C^\infty(T_HM).
    \]
    Similarly, the Lie derivative by $Z$ acts on distributions $\fbb\in \Ec'(T_HM)$ by
    \[
        \langle \Lc_Z\fbb,f\rangle  = \langle \fbb,-Zf\rangle,\quad \fbb \in \Ec'(T_HM),\; f \in C^\infty(T_HM)
    \]
    and on $r$-fibred distributions $\Ec'_r(T_HM)$ by the same formula
    \[
        (\Lc_Z\fbb,f) = (\fbb,-Zf),\quad \fbb\in \Ec'_r(T_HM),\; f \in C^\infty(T_HM).
    \]
    This overlap in notation should not lead to confusion, because
    \begin{equation}\label{overlap_of_notation}
        (\Lc_Z\fbb)_0 = \Lc_Z(\fbb_0),\quad \fbb\in \Ec'_r(\Tb_HM).
    \end{equation}
    Here, the symbol $\Lc_Z$ on the left is the derivation on $\Ec'_r(\Tb_HM)$ while the $\Lc_Z$ on the right is the derivation on $\Ec'_r(T_HM),$ and $\fbb_0\in \Ec'_r(T_HM)$ is the restriction to $h=0$ from Definition \ref{restriction}.

    It will be important for us that any $r$-fibred distribution $\Pb \in \Ec'_r(\Tb_HM)$ can be reconstructed from $(\Lc_Z-m)^{k+1}\Pb,$ for any complex number $m$ with negative real part and any $k\geq 0.$ This reconstruction uses the fact that elements of $\Ec'_r(\Tb_HM)$ have $r$-proper support in an essential way.
    \begin{lemma}\label{integral_representation_general}
        Let $m\in \Cplx.$ For all $\Pb \in \Ec'_r(\Tb_HM)$ and $t>0,$ we have
        \begin{equation}\label{cocycle_integral_representation}
            \Pb-t^{-m}(\alpha_{t})_*\Pb = \int_t^1 (\alpha_{\lambda})_*((\Lc_Z-m)\Pb)\frac{d\lambda}{\lambda^{m+1}}.
        \end{equation}
        Here, the integral is taken in the bounded convergence topology. If $\Re(m)<0,$ then
        \begin{equation}\label{negative_order_integral_representation}
            \Pb = \int_0^1 (\alpha_{\lambda})_*((\Lc_Z-m)\Pb)\,\frac{d\lambda}{\lambda^{m+1}}.
        \end{equation}
        More generally, for all $k\geq 0$ we have
        \begin{align}
            \Pb &= \sum_{j=0}^k \frac{1}{j!}(-\log t)^jt^{-m}(\alpha_{t})_*((\Lc_Z-m)^j\Pb)\nonumber\\
                &\quad+\frac{1}{k!}\int_t^1 (-\log\lambda)^{k}(\alpha_{\lambda})_*((\Lc_Z-m)^{k+1}\Pb)\frac{d\lambda}{\lambda^{m+1}}\label{full_kth_order_integral_representation}
        \end{align}
        and for $\Re(m)<0,$
        \begin{equation}\label{kth_order_integral_representation}
            \Pb = \frac{1}{k!}\int_0^1 (-\log\lambda)^{k}(\alpha_{\lambda})_*((\Lc_Z-m)^{k+1}\Pb)\frac{d\lambda}{\lambda^{m+1}}.
        \end{equation}
    \end{lemma}
    \begin{proof}
        Denote $\kbb = (\Lc_Z-m)\Pb.$ By Lemma \ref{smoothness_and_boundedness}.\eqref{smoothness}, the function
        \[
            \lambda\mapsto (\alpha_{\lambda})_*\kbb
        \]
        is smooth as a map from $\Rl_+$ into $\Ec'_r(\Tb_HM)$ with the bounded convergence topology.
        Hence, for $0<\lambda_0<\lambda_1<\infty,$ the integral
        \[
            \int_{\lambda_0}^{\lambda_1} (\alpha_{\lambda})_*\kbb \frac{d\lambda}{\lambda^{m+1}}
        \]
        is meaningful as an $\Ec'_r(\Tb_HM,\dens_r)$-valued integral, see \cite{ThomasTAMS75} and also \cite[Appendix B and Appendix C]{HLP2018}.

        It follows from Thomas' vector valued fundamental theorem of calculus for all $0<\lambda_0<\lambda_1 <\infty$ that
        \[
            \lambda_1^{-m}(\alpha_{\lambda_1})_*\Pb-\lambda_0^{-m}(\alpha_{\lambda_0})_*\Pb = \int_{\lambda_0}^{\lambda_1} (\alpha_{\lambda})_*((\Lc_Z-m)\Pb) \frac{d\lambda}{\lambda^{m+1}}.
        \]
        Taking $\lambda_1=1$ and $\lambda_0=t$ delivers \eqref{cocycle_integral_representation}.

        By Lemma \ref{smoothness_and_boundedness}.\eqref{boundedness}, when $\Re(m)<0$ the integral
        \[
            \int_{0}^1 (\alpha_{\lambda})_*\kbb \frac{d\lambda}{\lambda^{m+1}}
        \]
        converges in the topology of bounded convergence of $\Ec'_r(\Tb_HM),$ see \cite[Theorem 3]{ThomasTAMS75}. Thus, we have
        \[
            \Pb-\lim_{t\to 0}t^{-m}(\alpha_t)_*\Pb = \int_0^1 (\alpha_{\lambda})_*\kbb \frac{d\lambda}{\lambda^{m+1}}
        \]
        where the limit is in the bounded convergence topology. Since $\lim_{t\to 0}t^{-m}(\alpha_t)_*\Pb$ must be homogeneous of order $m$ with $\Re(m)<0$ and is properly supported, we have $\lim_{t\to 0}t^{-m}(\alpha_t)^*\Pb=0.$ This proves \eqref{negative_order_integral_representation}.

        Finally, \eqref{kth_order_integral_representation} follows from \eqref{negative_order_integral_representation} iterated $k$ times and a change of variable, using the
        integral formula
        \[
            \int_{\mu}^1 \lambda^{-1}(-\log(\lambda))^{k-1}\,d\lambda = \frac{1}{k}(-\log(\mu))^k.
        \]
    \end{proof}
    The identities \eqref{negative_order_integral_representation} and \eqref{kth_order_integral_representation} are essentially the integral formula
    \[
        (\Lc_Z-m)^{-k-1} = \frac{1}{k!}\int_{0}^1 (-\log\lambda)^k\lambda^{\Lc_Z-m}\frac{d\lambda}{\lambda}
    \]
    with the understanding that $\lambda^{\Lc_Z} = (\alpha_{\lambda})_*.$ The important point in the proof is to show that the integral converges.

\subsection{The van Erp-Yuncken calculus}
    We will adopt a slightly different perspective to \cite{vanErpYuncken2019} and define pseudodifferential families using the zoom vector field $Z.$
    \begin{definition}\label{differential_psido_definition}
        Let $m\in \Cplx.$ An $r$-fibred distribution $\Pb\in \Ec'_r(\Tb_HM)$ is called a \emph{pseudodifferential family of order} $m\in \Cplx$ if
        \[
            (\Lc_Z-m)\Pb \in C^\infty_p(\Tb_HM,\dens_r).
        \]
        Denote the space of pseudodifferential families of order $m$ as $\Psib^m_H(M).$
    \end{definition}
    Note that it is an immediate consequence of the Leibniz rule that
    \[
        \Psib^n_H(M)\ast \Psib^m_H(M)\subseteq \Psib^{m+n}_H(M)
    \]
    where $\ast$ is the groupoid convolution on $\Tb_HM.$

    Definition \ref{differential_psido_definition} is consistent with that of \cite{vanErpYuncken2019}, due to the following lemma.
    \begin{lemma}\label{consistency}
        We have $\Pb \in \Psib^m_H(M)$ if and only if for all $\lambda>0,$ we have
        \[
            \fbb(\lambda) := \lambda^{-m}(\alpha_{\lambda})_*\Pb-\Pb \in C^\infty_p(\Tb_HM,\dens_r).
        \]
    \end{lemma}
    \begin{proof}
        If $\Pb \in \Psib^m_H(M),$ then \eqref{cocycle_integral_representation} delivers
        \[
            \fbb(\lambda) = \int_{1}^\lambda (\alpha_{t})_*\kbb \frac{dt}{t^{m+1}}
        \]
        where $\kbb \in C^\infty_p(\Tb_HM,\dens_r).$ This integral converges in the LF space $C^\infty_p(\Tb_HM,\dens_r),$ since the function $t\mapsto (\alpha_t)_*\kbb$ is smooth as a $C^\infty_p(\Tb_HM,\dens_r)$-valued function. It follows that $\fbb(\lambda)\in C^\infty_p(\Tb_HM,\dens_r).$

        Conversely, if $\fbb(\lambda)\in C^\infty_p(\Tb_HM,\dens_r),$ then $\lambda\to \fbb(\lambda)$ is smooth as a function from $\Rl_+$ into $C^\infty_p(\Tb_HM,\dens_r)$ \cite[Lemma 21]{vanErpYuncken2019}. Differentiating with respect to $\lambda$ at $\lambda=1$ delivers $(\Lc_Z-m)\Pb \in C^\infty_p(\Tb_HM,\dens_r).$
    \end{proof}

    The function $\fbb$ corresponding to a pseudodifferential kernel of order $m$ from Lemma \ref{consistency} is called the corresponding cocycle, and satisfies the identity
    \begin{equation}\label{cocycle_identity}
        \fbb(\lambda\mu) = \fbb(\lambda)+\lambda^{-m}(\alpha_{\lambda})_*\fbb(\mu),\quad \lambda,\mu>0.
    \end{equation}

    With this preparation we recall van Erp--Yuncken's definition of their pseudodifferential calculus.
    \begin{definition}\label{psido_operator_definition}
        Let $(M,H)$ be a Carnot manifold. Let $r:M\times M\to M$ be the first component projection. An $r$-fibred distribution $P\in \Ec'_{r}(M\times M)$ is called a pseudodifferential operator of order $m\in \Cplx$
        if there exists $\Pb\in \Psib_H^m(M)$ such that $P=\Pb_1.$ Denote the space of pseudodifferential operators of order $m$ as $\Psi^m_H(M).$

        We denote $\Psi^{-\infty}_H(M):=\bigcap_{m\in \Cplx} \Psi^m_{H}(X).$ The set of all pseudodifferential operators is denoted $\Psi_H(M) := \bigcup_{m\in \Cplx} \Psi_H^m(M).$
    \end{definition}
    Van Erp and Yuncken have proved that $\Psi_H(M)$ has the desirable properties of a pseudodifferential calculus. In particular, $\Psi_H(M)$ is closed under operator composition, taking adjoints and consists of distributions whose wave-front set is conormal to the diagonal submanifold in $M\times M.$ In the trivially filtered case this construction reproduces the algebra of properly supported classical pseudodifferential operators on $M.$

    Note that $\Psi^{-\infty}_H(M)=C^\infty_p(M\times M,\dens_{r})$ is exactly the space of properly supported smoothing operators \cite[Corollary 53]{vanErpYuncken2019}.

\subsection{The log-polyhomogeneous calculus}\label{log_poly_subsection}
    As a straightforward extension of Definition \ref{differential_psido_definition}, we define order $k$ log-polyhomogeneous pseudodifferential familes as follows:
    \begin{definition}
        Let $k\geq 0$ be an integer, and let $m\in \Cplx.$ An $r$-fibred distribution $\Pb \in \Ec'_r(\Tb_HM)$ is called a $k$-log-polyhomogeneous pseudodifferential familiy if
        \[
            (\Lc_Z-m)^{k+1}\Pb \in C^\infty_p(\Tb_HM,\dens_r).
        \]
        The space of $k$-log-polyhomogeneous pseudodifferential families is denoted $\Psib^{m,k}_H(M).$
        An $r$-fibred distribution $P \in \Ec_r'(M\times M)$ is called a $k$-log-polyhomogeneous pseudodifferential operator of order $m$ if there exists $\Pb \in \Psib^{m,k}_H(M)$ for which $P = \Pb_1.$

        We will denote
        \[
            \Psib^{m,-1}_H(M) := C^\infty_p(\Tb_HM,\dens_r)
        \]
        and $\Psi^{m,-1}_H(M) = C^\infty(M\times M,\dens_r).$
    \end{definition}
    The fact that this definition reproduces Lesch's $k$-log-polyhomogeneous pseudodifferential operators when the filtration $H$ is trivial is essentially Theorem \ref{logpolyhomogeneous_characterisation} in the appendix.

    It is immediate that
    \[
        \Psib^{m,k-1}_H(M) \subseteq \Psib^{m,k}_H(M),\quad \Psi^{m,k-1}_H(M)\subseteq \Psi^{m,k}_H(M),\quad k\geq 0,\; m\in \Cplx.
    \]
    Observe that we could have defined $\Psib^{m,k}_{H}(M)$ inductively by
    \[
        \Psib^{m,k}_H(M) = \{\Pb \in \Ec_r'(\Tb_HM)\;:\; (\Lc_Z-m)\Pb \in \Psib^{m,k-1}(\Tb_HM)\},\quad k\geq 0.
    \]

    \begin{remark}
        We have $\Pb \in \Psib^{m,k}_H(M)$ if and only if for any $(k+1)$-tuple $\lambda_0,\ldots,\lambda_k,$ we have
        \[
            \sum_{n=0}^k \binom{k}{n}(-1)^{n-k}\lambda_n^{-m}(\alpha_{\lambda_n})_*\Pb \in C^{\infty}_p(\Tb_HM,\dens_r).
        \]
    \end{remark}


    The calculus $\Psi^{m,k}_{H}(M)$ is bigraded by $m$ and $k,$ as the following lemma shows.
    \begin{lemma}
        For all $m,n\in \Cplx$ and $k,\ell\geq -1,$ we have
        \[
            \Psib^{m,k}_H(M)\ast \Psib^{n,\ell}_H(M) \subseteq \Psib^{m+n,k+\ell}_H(M),\quad n,m\in \Cplx,\; k,\ell \geq 0
        \]
        and
        \[
            \Psi^{m,k}_H(M)\cdot \Psi^{n,\ell}_H(M) \subseteq \Psi^{m+n,k+\ell}_H(M)
        \]
        where $\cdot$ is the composition of operators.
    \end{lemma}
    \begin{proof}
        Let $\Pb\in \Psib^{m,k}_H(M)$ and $\Qb \in \Psib^{n,k}_H(M).$ By the Leibniz rule
        \begin{align*}
            (\Lc_Z-m-n)^{k+\ell+1}(\Pb\ast \Qb) &= \sum_{j=0}^{k+\ell+1} \binom{k+\ell+1}{j} ((\Lc_Z-m)^{k+\ell+1-j}\Pb)\ast ((\Lc_Z-n)^{j}\Qb)\\
                                                &= \sum_{j=0}^{\ell} \binom{k+\ell+1}{\ell-j}((\Lc_Z-m)^{k+1+j}\Pb)\ast ((\Lc_Z-n)^{\ell-j}\Qb)\\
                                                &\quad + \sum_{j=0}^{k} \binom{k+\ell+1}{\ell+1+j}((\Lc_Z-m)^{k-j}\Pb)\ast ((\Lc_Z-n)^{\ell+1+j}\Qb).
        \end{align*}
        In the first sum, the left factors are in the ideal $C^{\infty}_p(\Tb_HM,\dens_r),$ while in the second sum the right factors are in the ideal $C^\infty_p(\Tb_HM,\dens_r).$ So we conclude
        \[
            (\Lc_Z-m-n)^{k+\ell+1}(\Pb\ast \Qb) \in C^\infty_p(\Tb_HM,\dens_r).
        \]
        Since
        \[
            (\Pb \ast \Qb)_1 = \Pb_1\cdot \Qb_1
        \]
        the second statement follows at once.
    \end{proof}

\subsection{The principal symbol mapping and $H$-elliptic operators}
    By definition, an operator $P$ belongs to $\Psi^{m,k}_H(M)$ if there exists $\Pb\in \Psib^{m,k}_H(M)$ for which $P = \Pb_1.$ The choice $\Pb$ is not unique -- we can easily find nonzero $\Qb\in C^{\infty}_p(\Tb_HM,\dens_r)$ with $\Qb_1=0,$ and hence $(\Pb+\Qb)_1=P.$

%

    \begin{definition}
        Let
        \[
            \Sigmab^{m,k}_H(M) := \{T \in \Ec'_r(T_HM)\;:\;(\Lc_Z-m)^{k+1}T \in C^\infty_p(T_HM,\dens_r)\},\quad k\geq -1
        \]
        and
        \[
            \Sigma^{m,k}_H(M) := \Sigmab^{m,k}_H(M)/\Sigmab^{m,k-1}_H(M),\quad k\geq 0.
        \]
    \end{definition}

    The following lemma is the log-polyhomogeneous version of \cite[Lemma 36]{vanErpYuncken2019}, and can be proved identically using exponential cut-offs.
    \begin{lemma}\label{surjectivity_of_principal_symbol}
        For any $T \in \Sigmab^{m,k}_H(M),$ there exists $\Pb \in \Psib^{m,k}_H(M)$ such that $\Pb_0=T.$
    \end{lemma}
    Lemma \ref{surjectivity_of_principal_symbol} implies that if $\iota_m$ is the ``multiplication by $h$" map
    \[
        \iota_h:\Psib^{m-1,k}_H(M)\to \Psib^{m,k}_H(M)
    \]
    then $\iota_h$ induces an exact sequence
    \[
        0\rightarrow \Psib^{m-1,k}_H(M)\rightarrow \Psib^{m,k}_H(M)\rightarrow \Sigmab^{m,k}_H(M)\rightarrow 0.
    \]
    Standard homological algebra (the so-called ``nine lemma") implies that there is an induced exact sequence
    \[
        0\rightarrow \Psib^{m-1,k}_H(M)/\Psib^{m-1,k-1}_H(M)\rightarrow \Psib^{m,k}_H(M)/\Psib^{m,k-1}_H(M)\rightarrow \Sigma^{m,k}_H(M)\rightarrow 0.
    \]

%

    The following is the log-polyhomogeneous extension of \cite[Lemma 32]{vanErpYuncken2019}.
    \begin{lemma}\label{uniqueness_of_extension}
        Let $\Pb\in \Psib^{m,k}_H(M),$ for $m\in \Cplx$ and $k\geq 0.$ If $\Pb_1 \in \Psi^{m,k-1}_H(M),$ then ${\Pb_0 \in \Sigmab_H^{m,k-1}(M).}$
    \end{lemma}
    \begin{proof}
        By definition, $\Pb_1\in \Psi^{m,k-1}_H(M)$ if and only if there exists $\Qb \in \Psib^{m,k-1}_H(M)$ with $\Pb_1=\Qb_1.$ Replacing $\Pb$ by $\Pb-\Qb,$ it suffices to consider the case $\Pb_1=0.$ In local exponential coordinates, we have
        \[
            \Pb(x,z,h) = P(x,z,h)|dz|
        \]
        where $|dz|$ is the Lebesgue density. Denote $f_1,\ldots,f_{k+1}$ the distributions such that
        \[
            ((\Lc_Z-m)^{j}\Pb)(x,z,h) = f_j(x,z,h)dz.
        \]
        We use the integral representation \eqref{full_kth_order_integral_representation}, which in this notation gives
        \[
            P(x,z,h)-\sum_{j=0}^k \frac{1}{j!}(-\log t)^jt^{-m-Q}f_j(x,\delta_{t}^{-1}z,th) = \frac{1}{k!}\int_t^1 (-\log \lambda)^{k}\lambda^{-m-Q-1}f_{k+1}(x,\delta_{\lambda}^{-1}z,\lambda h)\,d\lambda.
        \]
        Taking $h=1$ (and using the assumption that $P(x,z,1)=0$) it follows that
        \[
            \sum_{j=0}^k \frac{1}{j!}(-\log t)^jf_j(x,z,t) = -\frac{1}{k!}\int_1^{1/t} (-\log(t\lambda))^k\lambda^{-m-Q-1}f_{k+1}(x,\delta_{\lambda}^{-1}z,t\lambda)\,d\lambda.
        \]
        This delivers
        \[
            f_k(x,z,t) = -\int_1^{1/t} (1+\frac{\log t}{\log \lambda})^k\lambda^{-m-Q-1}f_{k+1}(x,\delta_{\lambda}^{-1}z,t\lambda)\,d\lambda+o(1)
        \]
        where the $o(1)$ term vanishes in the topology of $\Ec'_r(T_HM)$ as $t\to 0.$ The integral term forms a bounded family in $C^{\infty}(K)$ for any $r$-proper subset $K$ of $T_HM,$ which implies by the Arzela-Ascoli theorem that there exists a limit point as $t\to 0.$ Since $f_k(x,z,t)$ converges in the sense of distribution as $t\to 0,$ we conclude that $f_k(\cdot,\cdot,0)|_K\in C^{\infty}(K).$
        That is,
        \[
            \Pb_0 \in \Sigmab_H^{m,k-1}(M).
        \]
%
    \end{proof}

    Lemma \ref{uniqueness_of_extension} implies the following theorem, which asserts that $\Sigma^{m,k}_H(M)$ is the correct principal symbol space for $\Psi^{m,k}_H(M).$
    \begin{theorem}
        There is a short exact sequence
        \[
            0\rightarrow \Psi^{m-1,k}_H(M)\rightarrow \Psi^{m,k}_H(M)\rightarrow \Sigma^{m,k}_H(M)\rightarrow 0
        \]
        where the left arrow is the inclusion and the right arrow is the map
        \[
            \sigma_m(P) := \Pb_0+\Sigmab^{m,k-1}_H(M)
        \]
        where $\Pb \in \Psib^{m,k}_H(M)$ is such that $\Pb_1=P.$
    \end{theorem}

%
%
%

\subsection{The local trace}
    We will make significant use of the ``local trace" or ``restriction to the unit space" map that restricts a distribution on the $H$-tangent groupoid to a distribution on the unit space $M\times \Rl.$ The unit space embedding
    \[
        \iota:M\times \Rl\to \Tb_HM
    \]
    is given by
    \[
        \iota(x,h) = \begin{cases} (x,x,h)\text{ if }h\neq 0\\
                                    ((x,0),0)\text{ if }h=0.
                      \end{cases}
    \]
    The unit space embedding pulls back to a restriction map on smooth functions
    \[
        \iota^*:C^\infty(\Tb_HM)\to C^\infty(M\times \Rl).
    \]
    What we need is not $\iota^*$ but instead the corresponding map on the bundle of densities tangent to the $r$-fibres.
    \begin{definition}
        Let $\evz$ denote the restriction to the image of $\iota.$ That is,
        \[
            \evz:C^\infty(\Tb_HM,\dens_r)\to C^\infty(M\times \Rl,\iota^*\dens_r)
        \]
        where $\iota^*\dens_r$ is the pullback of the bundle $\dens_r$ to $M\times \Rl$ along the unit space embedding $\iota.$
    \end{definition}
    To make the definition of $\evz$ more useful, we need to be more explicit about the bundle $\iota^*\dens_r\to M\times \Rl.$
    The splitting $\psi$ (as in Subsection \ref{subsec_groupoid_def}) induces an isomorphism of $TM$ with the vertical sub-bundle $V\tf_HM\subset T(\tf_HM),$
    Let $\ell:M\times \Rl\to M$ be the projection onto the first component. The isomorphism $TM\approx V(\tf_HM)$ induces an isomorphism
    \[
        j_{\psi}:\ell^*TM\to T(\tf_HM\times \Rl) \approx T(\tf_HM)\oplus T\Rl
    \]
    with the property that $j_{\psi}$ is an isomorphism onto the subspace
    \[
        V(\tf_HM)\oplus \{0\}\subset T(\tf_HM\times \Rl).
    \]

    \begin{lemma}\label{bundle_identification}
        Let
        \[
            j_{\psi}:\ell^*(TM)\to T(\tf_HM\times \Rl)
        \]
        be the embedding described above. Let
        \[
            \Gexp^{\psi}:\tf_HM\times \Rl \to \Tb_HM
        \]
        be a global exponential coordinate map. Then $j_{\psi}$ is an isomorphism
        \[
            j_{\psi}:\ell^*(TM)\to (\Gexp^{\psi})^*\ker(dr)
        \]
        on a neighbourhood of $M\times \{0\}$ in $M\times \Rl.$
    \end{lemma}
    \begin{proof}
        In the global exponential coordinates, the range map is the projection
        \[
            ((x,z),h)\to (x,h),\quad x\in M,\, z\in \tf_HM.
        \]
        and so the kernel of $dr$ is the subbundle of $T(\tf_HM\times \Rl) = T(\tf_HM)\oplus T\Rl$ given by
        \[
            V(\tf_HM)\oplus \{0\}.
        \]
        So, in these coordinates $j_{\psi}$ is an isomorphism of $\ell^*(TM)$ with $\ker(dr)$ by construction.
    \end{proof}

    Recall that a splitting is a bundle isomorphism
    \[
        \psi:\tf_HM\to TM
    \]
    with the property that $\psi|_{\tf_HM^k}$ is right-inverse to the quotient map $H_k\to H_{k-1}$ for all $k=1,...,N.$ Observe that if $\psi_1,\psi_2$ are two splittings,
    then $\psi_2^{-1}\psi_1$ is upper block-diagonal in the sense that
    \[
        \psi_2^{-1}\psi_1(\tf_HM^k) \subset \tf_HM^k\oplus \tf_HM^{k-1}\oplus\cdots \tf_HM^1,\quad k=1,\ldots,N.
    \]
    Moreover since $\psi_1$ is right-inverse to the quotient $H_k\to H_{k-1},$ the compression of $\psi_2^{-1}\psi_1$ to $\tf_HM^k$ is the identity map. It follows from these facts that the determinant of $\psi_2^{-1}\psi_1$ is identically $1.$ Hence the isomorphism in Lemma \ref{bundle_identification} induces an isomorphism
    \[
        \ell^*\dens M\approx \dens \ker(dr)
    \]
    which is independent of the splitting $\psi.$

    Henceforth we will identify $\evz$ with the associated map
    \[
        \evz:C^\infty(\Tb_HM,\dens_r)\to C^\infty(M\times \Rl,\ell^*\dens M).
    \]
    We call $\evz$ the \emph{local trace}. For $h\in \Rl,$ we denote by $\evz_h$ the evaluation of $\evz$ at $h\in \Rl,$ so that
    \[
        \evz_h:C^\infty(\Tb_HM,\dens_r)\to C^\infty(M,\dens M).
    \]
    Moreover $\evz_h$ is supported on the $h$-fibre of $\Tb_HM,$ so that the family $\{\evz_h\}_{h\in \Rl}$ interpolates between the local trace of an integral operator for $h\neq 0,$
    and the Liouville trace on the convolution algebra of $T_HM$ at $h=0.$

    The following relation is fundamental.
    \begin{lemma}\label{trace_scaling_lemma}
        Let $\fbb \in C^\infty_p(\Tb_HM,\dens_r).$ For all $\lambda>0$ and $h\in \Rl,$ we have
        \[
            \evz_h((\alpha_{\lambda})_*\fbb) = \lambda^{-\dhom}\evz_{\lambda h}(\fbb).
        \]
    \end{lemma}
    \begin{proof}
        Let $U\times V\times \Rl$ be a local exponential coordinate chart, where $U\subset M.$ In these coordinates, the unit space inclusion is
        \[
            \iota:U\times \Rl\to U\times V\times \Rl,\quad \iota(x,h) = (x,0,h)
        \]
        and in this chart we can write
        \[
            \fbb(x,z,h) = f(x,z,h)|dz|
        \]
        for some smooth function $f$ on $U\times V\times\Rl,$ where $dz$ is a density representing Lebesgue measure on $V.$ We have
        \[
            (\alpha_{\lambda})_*\fbb(x,z,h) = \lambda^{-\dhom}f(x,\delta_{\lambda}^{-1}z,\lambda h)|dz|
        \]
        hence
        \[
            (\iota^*\alpha_{\lambda})_*\fbb)(x,h) = \lambda^{-\dhom}f(x,0,\lambda h)\iota^*|dz|.
        \]
        On the other hand,
        \[
            \iota^*\fbb(x,\lambda h) = f(x,0,\lambda h)\iota^*|dz|.
        \]
        so we have
        \[
            \evz_h((\alpha_{\lambda})_*\fbb) = \lambda^{-\dhom}f(x,0,\lambda h)\iota^*|dz| = \lambda^{-\dhom}\evz_{\lambda h}(\fbb).
        \]
    \end{proof}

    Note that Lemma \ref{trace_scaling_lemma} includes the fact that
    \[
        \evz_0((\alpha_\lambda)_*\fbb) = \lambda^{-\dhom}\evz_0(\fbb),\quad \fbb\in C^\infty_p(\Tb_HM,\dens_r).
    \]
    Differentiating this relation with respect to $\lambda$ delivers the identity
    \[
        \evz_0(\Lc_Z\fbb) = -\dhom \evz_0(\fbb),\quad \fbb\in C^\infty_p(\Tb_HM,\dens_r).
    \]
    More generally, taking $\lambda=1+\varepsilon$ and identifying coefficients of $\varepsilon$ in the Taylor expansion of both sies of \eqref{trace_scaling_lemma}, for all $k\geq 1$ we have
    \begin{equation}\label{higher_derivative_relation}
        h^k\evz^{(k)}_h(\fbb) = \evz_h((\Lc_Z+\dhom)(\Lc_Z+\dhom-1)\cdots(\Lc_Z+\dhom-k-1)\fbb).
    \end{equation}

    The local trace map extends to pseudodifferential kernels with sufficiently negative order.
    \begin{lemma}\label{trace_on_trace_class}
        If $\Re(m)<-\dhom$ and $k\geq 0$ there is a unique continuous map
        \[
            \evz:\Psib^{m,k}_H(M)\to C^0(M\times \Rl,\dens M)
        \]
        extending the definition of $\evz$ on $C^\infty_p(\Tb_HM,\dens_r).$ Similarly, for $h\in \Rl,$ we have a continuous linear map
        \[
            \evz_h:\Psib^{m,k}_H(M)\to C^0(M,\dens M).
        \]
    \end{lemma}
    \begin{proof}
        Let $\kbb = (\Lc_Z-m)^{k+1}\Pb.$ By \eqref{kth_order_integral_representation}, we have
        \[
            \Pb = \frac{1}{k!}\int_0^1 (-\log\lambda)^{k}(\alpha_{\lambda})_*\kbb\frac{d\lambda}{\lambda^{m+1}}.
        \]
        By Lemma \ref{smoothness_and_boundedness}.\eqref{boundedness_2}, this integral converges in $C^0(\Tb_HM,\dens_r).$
        We define
        \[
            \evz_h(\Pb) := \frac{1}{k!}\int_0^1 (-\log\lambda)^k\tr_h((\alpha_{\lambda})_*\kbb)\frac{d\lambda}{\lambda^{m+1}}.
        \]
        By Lemma \ref{trace_scaling_lemma}, the integral is the same as
        \[
            \frac{1}{k!}\int_0^1 (-\log\lambda)^k\evz_{\lambda h}(\kbb) \frac{d\lambda}{\lambda^{m+\dhom+1}}.
        \]
        Since $\Re(m)<-\dhom,$ this integral converges in the Fr\'echet topology of $C^0(M,\dens M),$ and defines a continuous function of $h.$
    \end{proof}

    It will be important that the local trace is defined not just for pseudodifferential families, but also for pseudodifferential operators.
    \begin{lemma}
        Let $\Re(m)<-\dhom,$ and let $\Pb\in \Psib^{m}(\Tb_HM,\dens_r)$ be such that $\Pb_1 = 0.$ Then
        \[
            \evz_1(\Pb) = 0.
        \]
    \end{lemma}
    \begin{proof}
        From \cite[Lemma 32]{vanErpYuncken2019}, $\Pb_1=0$ implies that $\Pb \in C^\infty_p(\Tb_HM,\dens_r).$ Hence,
        \[
            \evz_1(\Pb) = \tr(\Pb_1)=0.
        \]

        Let $\kbb = (\Lc_Z-m)\Pb.$
        By \eqref{negative_order_integral_representation}, we have
        \[
            \Pb = \int_0^1 (\alpha_{\lambda})_*\kbb\frac{d\lambda}{\lambda^{m+1}}.
        \]
        So,
        \[
            \evz_1(\Pb) = \int_0^1 \evz_{\lambda}(\kbb)\frac{d\lambda}{\lambda^{\dhom+m+1}}.
        \]
        On the other hand, in $\Ec'_r(M\times M)$ we have
        \[
            0 = \Pb_1 = \int_0^1 \kbb_{\lambda}\frac{d\lambda}{\lambda^{m+1}}.
        \]
        Therefore
        \[
            \int_0^1 \tr_1(\kbb_{\lambda})\frac{d\lambda}{\lambda^{m+1}}=0.
        \]
    \end{proof}

    It follows from the above lemma that if $P\in \Psi^m_H(M)$ with $\Re(m)<-\dhom,$ then the value of
    \[
        \tr(P) := \evz_1(\Pb) \in C^0(M,\dens M).
    \]
    is independent of the choice of pseudodifferential family $\Pb \in \Psib^{m}_H(M)$ with $\Pb_1 = P.$ What this relation further implies is
    \[
        \evz_h(\Pb) = h^{\dhom}\evz_1(\Pb_h)
    \]
    where the restriction $\Pb_h$ is as in Definition \ref{restriction}.

    When $M$ is compact, or at least if $P$ is compactly supported on $M\times M,$ we can further integrate $\evz(P)$ over $M$ to define a global trace
    \begin{equation}\label{global_trace_on_operators}
        \Tr:\Psi^m_H(M)\to \Cplx.
    \end{equation}
    This global trace coincides with the Hilbert space trace of the realisation of $P$ as an operator on $L_2(M,\nu)$ for some choice of nowhere-vanishing density $\nu.$

    \begin{definition}\label{local_trace_germ_definition}
        For $\fbb\in C^\infty(\Tb_HM,\dens_r)$ and $j\geq 0,$ denote
        \[
            \evz_h^{(j)}(\fbb) := \frac{d^j}{dh^j}\evz_h(\fbb),\quad \fbb\in C^\infty(\Tb_HM,\dens_r).
        \]
    \end{definition}

    Although the local trace $\evz$ is only defined for operators whose order has real part less than $-\dhom,$ it will be important to re-express $\evz_h(\Pb)$ in a form that is meaningful for a larger class of operators.
    \begin{lemma}\label{trace_extension}
        Let $k\geq 0,$ and $\Pb \in \Psib^{m,k}_{H}(M),$ where $\Re(m)<-\dhom.$ For any $n\geq 1$ and $h\geq 0$ we have
        \begin{align*}
            \evz_h(\Pb) &= \sum_{j=0}^{n-1} \frac{h^j}{j!(j-m-\dhom)^{k+1}}\evz_0^{(j)}((\Lc_Z-m)^{k+1}\Pb)\\
            &\quad +\frac{h^{n}}{(n-1)!k!}\int_0^1\evz_{hu}^{(n)}((\Lc_Z-m)^{k+1}\Pb)u^{n-m-\dhom-1}\left(\int_u^1 \log(t/u)^{k}t^{m+\dhom-n}(1-t)^{n-1}\,dt\right)\,du.
        \end{align*}

        In particular, if $P \in \Psi^m_H(M)$ and $\Pb\in \Psib^m_H(M)$ is any pseudodifferential kernel with $\Pb_1=P,$ we have
        \begin{align*}
            \tr(P) &= \sum_{j=0}^{n-1} \frac{1}{j!(j-m-\dhom)}\evz_0^{(j)}((\Lc_Z-m)\Pb)\\
            &\quad +\frac{1}{(n-1)!}\int_0^1\evz_u^{(n)}((\Lc_Z-m)\Pb)u^{n-m-\dhom-1}\Beta(1-u;n,m+\dhom-n+1)\,du
        \end{align*}
        where
        \[
            \Beta(x;\alpha,\beta) = \int_0^x t^{\alpha-1}(1-t)^{\beta-1}\,dt
        \]
        is the incomplete Beta function.
    \end{lemma}
    \begin{proof}
        From Lemma \ref{trace_on_trace_class}, writing $\kbb = (\Lc_Z-m)^{k+1}\Pb$ we have
        \begin{equation}\label{integral_form_of_trace}
            \evz_h(\Pb) = \frac{1}{k!}\int_0^1 (-\log\lambda)^k\evz_{\lambda h}(\kbb) \frac{d\lambda}{\lambda^{m+\dhom+1}}.
        \end{equation}
        On the other hand, by Taylor's theorem,
        \[
            \evz_{\lambda h}(\kbb) = \sum_{j=0}^{n-1} \frac{\lambda^jh^j}{j!}\evz_0^{(j)}(\kbb)+\frac{h^{n}}{(n-1)!}\int_0^{\lambda} (\lambda -u)^{n-1}\evz_{hu}^{(n)}(\kbb)\,du.
        \]
        Substituting this into \eqref{integral_form_of_trace} yields
        \begin{align*}
            \evz_h(\Pb) &= \sum_{j=0}^{n-1} \frac{h^j}{j!k!}\int_0^1 (-\log\lambda)^k\lambda^{j-m-\dhom-1}\,d\lambda\,\evz_0^{(j)}(\kbb)\\
            &\quad+\frac{h^{n}}{(n-1)!k!}\int_0^1 \int_{0}^{\lambda} (-\log\lambda)^k(\lambda-u)^{n-1}\lambda^{-m-\dhom-1}\tr_{hu}^{(n)}(\kbb)\,du d\lambda.
        \end{align*}
        Exchanging the order of integration in the latter term,
        \begin{align*}
            \evz_h(\Pb) &= \sum_{j=0}^{n-1} \frac{h^j}{j!(j-m-\dhom)^{k+1}}\evz^{(j)}_0(\kbb)\\
            &\quad + \frac{h^{n}}{(n-1)!k!}\int_0^1 \evz_{hu}^{(n)}(\kbb)\left(\int_{u}^1 (-\log\lambda)^k(\lambda-u)^{n-1}\lambda^{-m-\dhom-1}\,d\lambda \right)\,du.
        \end{align*}
        For $k=0,$ the proof is complete on writing
        \[
            \int_{u}^1 (\lambda-u)^{n-1}\lambda^{-m-\dhom-1}\,d\lambda = u^{n-m-\dhom-1}\int_u^1 (-\log(u/t))^kt^{m+\dhom-n}(1-t)^{n-1}\,dt.
        \]
        In the case $k=0,$ this integral evaluates as $\Beta(1-u;n,m+\dhom-n+1).$
    \end{proof}

    \begin{remark}
        Note that if $\Re(m)<n-\dhom,$ the integral
        \[
            \int_0^1 \evz_{hu}^{(n)}(\kbb)u^{n-m-\dhom-1}\int_u^1 \log(t/u)^k t^{m+\dhom-n}(1-t)^{n-1}\,dtdu
        \]
        converges absolutely.

        This means that the expression for $\evz_h(\Pb)$ in Lemma \ref{trace_extension} makes sense provided that $\Re(m)<n-\dhom,$ excepting $m=-\dhom,-\dhom+1,\ldots,-\dhom+n-1.$ Taking $n$ large as necessary allows the definition of $\evz_h$ to be extended to pseudodifferential kernels whose order does not belong to $-\dhom+\Ntrl.$ In the trivially filtered case, this extension is precisely the (local) canonical trace, also known as the Kontsevich-Vishik trace.
    \end{remark}

\section{Holomorphic families of pseudodifferential operators}
    Dave and Haller have developed the basic theory of holomophic families of pseudodifferential operators in the van Erp--Yuncken calculus \cite[Section 6]{DaveHaller2020}. Here, we present the same idea from the perspective of Definition \ref{differential_psido_definition}.

    Recall that $C^\infty_p(\Tb_HM,\dens_r)$ is the space of $r,s$-properly supported $r$-fibred distributions on $\Tb_HM.$ This is a topological vector space when equipped with the strong topology induced by
    \[
        C^\infty_p(\Tb_HM,\dens_r) = \bigcup_{U\subset \Tb_H} C^\infty(U,\dens_r)
    \]
    where the union is over all subsets with $r,s$-proper closure.

    Let $\Omega\subseteq \Cplx$ be connected and open. A function
    \[
        f:\Omega\to C^\infty_p(\Tb_HM,\dens_r)
    \]
    is said to be holomorphic if for every open subset of $V\subset \Omega$ with compact closure, there exists an $r,s$-proper subset $U$ of $\Tb_HM$ such that the restriction of $f$ to $V$ is holomorphic as a function from $V$ to the Fr\'echet space $C^\infty(U,\dens_r).$

The holomorphic families version of Definition \ref{differential_psido_definition} is the following.
\begin{definition}
    Let $\Omega\subseteq \Cplx$ be open and connected, and let $\mu:\Omega\to \Cplx$ be a holomorphic function. Let $k\geq 0.$

    A family of $r$-fibred distributions $\Pb = \{\Pb(s)\}_{s\in \Omega}$ is said to be a $k$-log polyhomogeneous pseudodifferential family with order function $\mu$ if the following holds:
    \begin{enumerate}[{\rm (i)}]
        \item{} $s\mapsto \Pb(s)$ is holomorphic as an $\Ec'_r(\Tb_HM)$-valued function.
        \item{} $(\Lc_Z-\mu(s))^{k+1}\Pb(s) \in C^\infty_p(\Tb_HM,\dens_r)$ for all $s\in \Omega.$
    \end{enumerate}
    A family of pseudodifferential operators $P = \{P(s)\}_{s\in \Omega}$ is said to be a holomorphic pseudodifferential family with order function $\mu$ if there exists a holomorphic pseudodifferential kernel family $\Pb$ with order function $\mu$ such that $P(s) = \Pb(s)_1$ for all $s\in \Omega.$
    \end{definition}
It is straightforward to check that this is consistent with \cite[Section 7]{DaveHaller2020} in the case $k=0,$ using the same methods as the proof of Lemma \ref{consistency}.

The fundamental theorem about holomorphic families is that the local traces of holomorphic families admit a meromorphic continuation. The following is essentially \cite[Proposition 2]{DaveHaller2020}.
\begin{proposition}\label{where_are_the_poles}
    Let $\Pb$ be a $k$-log polyhomogeneous pseudodifferential family on the domain $\Omega\subseteq \Cplx$ with order function $\mu:\Omega\to \Cplx.$ If
    \[
        \{s\in \Omega\;:\;\Re(\mu(s))<-\dhom\}\neq \emptyset
    \]
    then for every $x \in M,$ the $C(M\times \Rl,\dens M)$-valued function
    \[
        s\mapsto \evz(\Pb(s)), \quad \Re(\mu(s))<-\dhom,\; h > 0
    \]
    admits a meromorphic continuation to $\Omega,$ with poles located at the points
    \[
        \mu^{-1}(\{-\dhom,-\dhom+1,-\dhom+2,-\dhom+3,\ldots\}).
    \]
    If $s_0 \in \Omega$ is such that $\mu(s)=-\dhom,$ and $\mu'(s_0)\neq 0,$ then there is a pole of order at most $k+1$ at $s=s_0,$ and for all $h\in \Rl$ we have
    \[
        \lim_{s\to s_0}(s-s_0)^{k+1}\evz_h(\Pb(s)) = \frac{1}{(-\mu'(s_0))^{k+1}}\evz_0((\Lc_Z+\dhom)^{k+1}(\Pb(s_0))).
    \]
\end{proposition}
\begin{proof}
    Let $s\in \Omega$ be such that $\Re(\mu(s))<-\dhom.$ By Lemma \ref{trace_extension}, for all $n\geq 1$ we have
    \begin{align*}
        &\evz_h(\Pb(s)) = \sum_{j=0}^{n-1} \frac{h^j}{j!(j-\mu(s)-\dhom)^{k+1}}\evz_0^{(j)}((\Lc_Z-\mu(s))^{k+1}\Pb(s))\\
    &\quad +\frac{h^n}{(n-1)!k!}\int_0^1\evz_{hu}^{(n)}((\Lc_Z-\mu(s))\Pb(s))u^{n-\mu(s)-\dhom-1}\left(\int_{u}^1\log(t/u)^{k}(1-t)^{n-1}t^{\mu(s)+\dhom-n}\,dt\right)\,du.
    \end{align*}
    Observe that the right hand side defines a meromorphic function in the set $\Re(\mu(s))<n-1-\dhom.$ The possible poles are only located at those points $s$ such that $\mu(s)=j-\dhom$ for some $1\leq j\leq n-2.$

    Since
    \[
        \Omega = \bigcup_{n\geq 1} \{s\;:\;\Re(\mu(s))<n-\dhom\}
    \]
    it follows that $\evz_h(\Pb)$ has a meromorphic extension to the domain $\Omega,$ with poles only located at those points $s$ such that $j-\dhom-\mu(s)=0$ for some $j\geq 0.$ The formula for the coefficient of the pole at those points where $\mu(s)=-\dhom$ is now immediate.
\end{proof}

\begin{remark}
    Although we will not need it, the proof of Proposition \ref{where_are_the_poles} also gives a formula for the residues of $\evz_h\Pb(s)$ at other values of $s.$ For $j\in \Ntrl,$
    let $s_j\in \Cplx$ be such that $\mu(s_j) = j-Q,$ and assume for simplicity that $\mu'(s_j)\neq 0$ and $k=0.$

    Then the pole of the meromorphic continuation of $\evz_h\Pb(s)$ at $s=s_0$ is either simple or removable, and
    \[
        \Res_{s=s_j} \evz_h\Pb(s) = -\frac{h^j}{j!\mu'(s_j)}\evz_0^{(j)}((\Lc_Z-j+\dhom)\Pb(s_j)).
    \]
    Recall that the notation $\evz_0^{(j)}$ was given in Definition \ref{local_trace_germ_definition}. This formula incidentally shows that the value of the residue of the meromorphic continuation of $\evz_h(\Pb(s))$ at $s_j$ depends only on $\Pb(s_j)$ and $\mu'(s_j).$

    By \eqref{higher_derivative_relation}, we can also write this as
    \[
        \Res_{s=s_j}\evz_h\Pb(s) = -\frac{(j+1)h^j}{\mu'(s_j)}\lim_{\varepsilon\to 0} \varepsilon^{-j}\evz_\varepsilon(\binom{\Lc_Z+\dhom}{j+1}\Pb(s_j)))
    \]
    where $\binom{x}{j+1}$ is the polynomial
    \[
        \binom{x}{j+1} = \frac{1}{(j+1)!}x(x-1)\cdots (x-j).
    \]
\end{remark}

Integrating over $M$ in the compact case delivers the following similar theorem.
\begin{corollary}\label{meromorphic_extension_of_trace}
    Let $\{P(s)\}_{s\in \Omega}$ be a holomorphic family of $k$-log-polyhomogeneous pseudodifferential operators with order function $\mu:\Omega\to \Cplx$ on a compact Carnot manifold $(M,H).$ If
    \[
        \{s\in \Omega\;:\;\Re(\mu(s))<-\dhom\}\neq \emptyset
    \]
    then the global trace
    \[
        s\mapsto \Tr(P(s)),\;\quad \Re(\mu(s))<-\dhom
    \]
    admits a meromorphic continuation to $\Omega,$ with poles located of order at most $k+1$ at the points
    \[
        \mu^{-1}(\{-\dhom,-\dhom+1,\ldots\})
    \]
    If $s_0\in \Omega$ is such that $\mu(s_0)=-\dhom,$ then
    \[
        \lim_{s\to s_0}(s-s_0)^{k+1}\Tr(P(s)) = \lim_{s\to s_0}\frac{s-s_0}{-\dhom-\mu(s)}\int_M \evz_0((\Lc_Z-\mu(s))^{k+1}\Pb(s))
    \]
    for any holomorphic family of pseudodifferential kernels $\Pb$ such that $P(s) = \Pb(s)_1$ for all $s\in \Omega.$
\end{corollary}

\section{The $k$th order residue}\label{residue_section}
Let $(M,H)$ be a compact Carnot manifold. In \cite{DaveHaller2020}, Dave-Haller defined the residue of an operator $A \in \Psi^{-\dhom}_H(M)$ as
\[
    \Res(A) := \Res_{s=0} \Tr(A(s))
\]
where $A(s)$ is any holomrphic family with $A(0) = A$ and order function $\mu(s)=-\dhom-s.$

From Corollary \ref{meromorphic_extension_of_trace} with $k=0,$ we can see that
\begin{equation}\label{DH_residue_as_evaluation_at_zero}
    \Res(A) = \int_M \evz_0((\Lc_Z+\dhom)\Ab)
\end{equation}
where $\Ab\in \Psib^{-\dhom}_H(M)$ is any pseudodifferential family with $\Ab_1=A.$ Couchet-Yuncken \cite{CouchetYuncken2024} defined a residue density $\res(A)\in C^{\infty}(M,\dens M),$ defined in the following way: let $\Ab\in \Psib^{-\dhom}_H(M)$ be such that $\Ab_1=A.$ By definition, for $\lambda>0,$ the function
\[
    \fbb(\lambda) = \lambda^{-\dhom}(\alpha_{\lambda})_*\Ab-\Ab
\]
belongs to $C^{\infty}_p(\Tb_HM,\dens_r).$ It turns out that there is a unique density $\res(A)$ such that
\[
    \res(A) := \frac{1}{\log(\lambda)}\evz_0(\fbb(\lambda)).
\]
for all $\lambda>1.$ The fact that $\evz_0(\fbb(\lambda))$ is a constant multiple of $\log(\lambda)$ follows from \eqref{cocycle_identity}.

Let $\kbb = (\Lc_Z+\dhom)\Ab.$ By Lemma \ref{consistency}, we have
\[
    \fbb(\lambda) = \int_1^\lambda \mu^{\dhom}(\alpha_{\mu})_*\kbb\frac{d\mu}{\mu}.
\]
Applying the local trace map and evaluating at $h=0$ gives
\[
    \evz_0\fbb(\lambda) = \evz_0\kbb\int_1^\lambda \frac{d\mu}{\mu} = \evz_0\kbb\log(\lambda).
\]
Dividing by $\log(\lambda)$ yields $\mathrm{res}(A) = \evz_0(\kbb).$ Therefore,
\begin{equation}\label{CY_residue_as_evaluation_at_zero}
    \res(A) = \evz_0((\Lc_Z+\dhom)\Ab).
\end{equation}
It follows from \eqref{CY_residue_as_evaluation_at_zero} that $\res(A)$ depends only on $\Ab_0,$ that is the principal cosymbol of $A.$ Indeed, \eqref{overlap_of_notation} implies that
\[
    \res(A) = \tau((\Lc_Z+\dhom)\Ab_0).
\]

\begin{definition}
    Let $k\geq 0.$ The $k$th order residue density is the density-valued functional
    \[
        \res_k:\Psi^{-\dhom,k}_H(M)\to C^{\infty}(M,\dens M)
    \]
    given by
    \[
        \res_k(A) = \evz_0((\Lc_Z+\dhom)^{k+1}\Ab)
    \]
    where $\Ab\in \Psib^{-\dhom,k}_H(M)$ is any pseudodifferential family with $\Ab_1=A.$

    We write
    \[
        \Res_k(A) := \int_M \res_k(A)
    \]
    provided that the integral converges.
\end{definition}
With this notation, Dave-Haller's residue $\Res(A)$ is written $\Res_0(A),$ and Couchet-Yuncken's residue density is $\res_0(A).$

\begin{lemma}
    Let $k\geq 0,$ and let $A \in \Psi^{-\dhom,k}_H(M).$ The definition of $\res_k(A)$ is independent of the choice of $\Ab\in \Psib^{-\dhom,k}_H(M)$ such that $\Ab_1=A.$
\end{lemma}
\begin{proof}
    Let $\Ab.\Bb\in \Psib^{-\dhom,k}_H(M)$ be such that $\Ab_1=\Bb_1=A.$ By Lemma \ref{uniqueness_of_extension}, we have $\Ab-\Bb \in \Psib^{-\dhom,k-1}_H(M).$ That is,
    \[
        (\Lc_Z+\dhom)^{k}(\Ab-\Bb) \in C^{\infty}_p(\Tb_HM,\dens_r).
    \]
    By Lemma \ref{trace_scaling_lemma}, it follows that
    \[
        \evz_0((\Lc_Z+\dhom)^{k+1}(\Ab-\Bb)) = 0.
    \]
\end{proof}

\begin{lemma}
    Let $k\geq 0,$ and let $A \in \Psi^{-\dhom,k}_H(M).$ The residue density $\res_k(A)$ depends only on $\sigma_{-\dhom,k}(A),$ and is computable by the identity
    \[
        \res_k(A) = \tau((\Lc_Z+\dhom)^{k+1}\sigma_{-\dhom,k}(A)).
    \]
\end{lemma}
This should be understood as an abuse of notation: by definition, $\sigma_{-\dhom,k}(A)$ belongs to the space
\[
    \Sigma^{-\dhom,k}_H(M) = \Sigmab^{-\dhom,k}_H(M)/\Sigmab^{-\dhom,k-1}_H(M).
\]
If $\sigma_{-\dhom,k}(A)$ is any representative of a coset of $\Sigmab^{-\dhom,k-1}_H(M),$ then $\tau((\Lc_Z+\dhom)^{k+1}\sigma_{-\dhom,k}(A))$ is the residue density $\res_{k}(A).$

\section{Singular value estimates of $\log$-polyhomogeneous $\psi$DO}
In this section, $(M,H)$ is a \emph{compact} Carnot manifold which we equip with a nowhere-vanishing density $\nu.$ We write $L_2(M)$ for the Hilbert space $L_2(M,\nu),$ and $\Ec_{p,k}=\Ec_{p,k}(L_2(M)).$
Distributions $K\in \Ec'_r(M\times M)$ are identified with operators from $C^{\infty}(M)$ to $C^\infty(M).$
\begin{lemma}\label{elementary_hump_estimates}
    Let $\kbb \in C^{\infty}_p(\Tb_HM,\dens_r).$ For $0<t<1,$ we have
    \begin{equation}\label{Hilbert_Schmidt_estimate}
        \|((\alpha_t)_*\kbb)_1\|_{\Lc_{2}} \leq C_{\kbb}t^{\frac{\dhom}{2}}
    \end{equation}
    and
    \begin{equation}\label{operator_norm_estimate}
        \|((\alpha_t)_*\kbb)_1\|_{\Lc_{\infty}} \leq C_{\kbb}.
    \end{equation}
\end{lemma}
\begin{proof}
    Since the manifold $M$ is compact, it suffices to work in a coordinate chart. In local exponential coordinates, we have
    \[
        \kbb(x,z,h) = f(x,z,h)|dz|,\quad (x,z,h)\in T_HM\times \Rl
    \]
    where $|dz|$ is the Lebesgue measure on the tangent space. The operator $((\alpha_t)_*\kbb)_1$ has integral kernel
    \[
        (\alpha_t)_*\kbb(x,y)=f(x,\delta_t^{-1}\exp^{-1}_x(y),t),\quad x,y\in M
    \]
    for $y$ sufficiently close to $x$ such that $\exp^{-1}_x(y)$ is defined.

    The Hilbert-Schmidt norm is the $L_2(M\times M)$-norm of the integral kernel, and the operator norm can be estimated above by the $L_{\infty}(M,L_1(M))$ norm of $\kbb$ and $\kbb^*.$ The upper bounds \eqref{Hilbert_Schmidt_estimate} and \eqref{operator_norm_estimate} follow from simple volume estimates.
\end{proof}

\begin{theorem}\label{schatten_class_theorem}
    For all $m>0$ and $k\geq 0,$ we have
    \[
        \Psi^{-m,k}_H(M)\subset \Ec_{\frac{\dhom}{m},k}.
    \]
\end{theorem}
\begin{proof}
    Initially assume that $m<\frac{\dhom}{2}.$
    Let $A\in \Psi^{-m,k}_H(M),$ and $\Ab\in \Psib^{m,k}_H(\Tb_HM)$ be such that $\Ab_1=A.$ Abbreviate $\kbb = (\Lc_Z+m)^{k+1}\Ab.$

    By Lemma \ref{negative_order_integral_representation}, we have
    \[
        \Ab = \frac{1}{k!}\int_0^1 \lambda^{m}(-\log\lambda)^{k}(\alpha_{\lambda})_*(\kbb)\frac{d\lambda}{\lambda}.
    \]
    Let $0<\mu<1,$ and write
    \begin{align*}
        A &= \frac{1}{k!}\int_0^{\mu} \lambda^{m}(-\log\lambda)^{k}((\alpha_{\lambda}))_*(\kbb))_1\frac{d\lambda}{\lambda}\\
            &+\frac{1}{k!}\int_{\mu}^1 \lambda^{m}(-\log\lambda)^k((\alpha_{\lambda})_*(\kbb))_1\frac{d\lambda}{\lambda}\\
            &=: A^{(1)}+A^{(2)}.
    \end{align*}
    By Lemma \ref{elementary_hump_estimates}, we have
    \[
        \|A^{(1)}\|_{\Lc_2} \lesssim \int_{\mu}^1 \lambda^{m-\frac{\dhom}{2}}(-\log\lambda)^k\frac{d\lambda}{\lambda} \lesssim \mu^{m-\frac{\dhom}{2}}(-\log\mu)^k
    \]
    and
    \[
        \|A^{(2)}\|_{\Lc_{\infty}}\lesssim \int_0^{\mu} \lambda^{m}(-\log\lambda)^k\frac{d\lambda}{\lambda} \lesssim \mu^m(-\log\mu)^k.
    \]
    Therefore
    \[
        \|A^{(1)}\|_{\Lc_2}+t\|A^{(2)}\|_{\Lc_{\infty}} \lesssim (\mu^{m-\frac{\dhom}{2}}+t\mu^m)(-\log\mu)^k.
    \]
    Choosing $\mu = t^{-\frac{2}{\dhom}}$ gives, via Lemma \ref{K_function_identity},
    \[
        \left(\int_0^{t^2}\mu(s,A)^2\,ds\right)^{1/2} \lesssim t^{1-\frac{2m}{\dhom}}(\log t)^k,\quad t>1.
    \]
    Since $\frac{\dhom}{m}>2,$ it follows from Lemma \ref{Hardy_type_inequality} that $A \in \Ec_{\frac{\dhom}{m},k}.$

    In the general case, let $P\geq 1$ be an $H$-elliptic operator of order $n.$ By Dave-Haller's Weyl asymptotics, we have $P^{-\frac{t}{np}}\in \Lc_{\frac{\dhom}{t},\infty}=\Ec_{p,0}.$ Therefore, for $A\in \Psi^{-m,k}_H(M),$ we have
    \[
        A = AP^{\frac{t}{n}}\cdot P^{-\frac{t}{n}} \in \Psi^{-m+t,k}\cdot \Ec_{\frac{Q}{t},0}.
    \]
    If we choose $t$ so that $0>-m+t> -\frac{\dhom}{2},$ then we get $\Psi^{-m+t,k}_H(M)\subset \Ec_{\frac{Q}{m-t},k},$ and hence $\Psi^{-m,k}_H(M)\subset \Ec_{\frac{Q}{m},k}$ by \eqref{holder_type}.
\end{proof}

\section{Universal measurability of van Erp-Yuncken pseudodifferential operators}\label{dave_haller_residue}
Throughout this section, we work in the same setting as Dave-Haller \cite{DaveHaller2020}: $(M,H)$ is a compact Carnot manifold, and $\Psi^m_H(M)$ denotes the order $m$ pseudodifferential operators in the van Erp-Yuncken sense.

Fix an order $m$ Rockland differential operator $P,$ and assume that $P\geq 1.$ Such operators always exist although we might have $m>2.$ Since Rockland operators are hypoelliptic, it is easy to see that $\mathrm{Ker}(1+P^*)=\{0\}$ and therefore $P$ has self-adjoint extension. Dave and Haller proved that the complex powers $P^{-s},$ defined via spectral theory, are Van Erp-Yuncken pseudodifferential operators and more specifically
\[
    P^{-s} \in \Psi^{-ms}_H(M).
\]
The main result of \cite{DaveHaller2020} is a Weyl formula for the eigenvalues of $P.$ In particular, $P^{-s}$ belongs to the trace class when $\Re(s)$ is sufficiently large. What is more, they proved the following:
\begin{theorem}\label{DH_analyticity_theorem}
    Let $(M,H)$ be a compact Carnot manifold, and let $\nu$ be a density on $M.$ Let $P$ be an order $m$ $H$-elliptic differential operator, and assume that $P\geq 1$ as an operator on $L_2(M,\nu).$ The family of complex powers
    \[
        \{P^{-s}\}_{s\in \Cplx}
    \]
    defined by spectral theory for operators on $L_2(M,\nu)$ is a pseudodifferential family with order function
    \[
        \mu(s) = -ms.
    \]
\end{theorem}

The following theorem is a combination of Theorem \ref{DH_analyticity_theorem} and Proposition \ref{where_are_the_poles}. This is also contained within \cite[Proposition 2]{DaveHaller2020}.
\begin{corollary}\label{zeta_function_analyticity}
    Let $\nu$ be a density on $M,$ and let $A \in \Psi^0_H(M).$ Let $P$ be an order $m$ $H$-elliptic differential operator on $M$ such that $P\geq 1$ with respect to $\nu.$ Identifying $A$ with an operator on $L_2(M,\nu),$ The zeta-function
    \[
        \zeta_{A,P}(s) := \Tr(AP^{-s}),\quad \re(s)>\frac{\dhom}{m}
    \]
    admits a meromorphic continuation to $\Cplx,$ with at most simple poles located at the points $\{\frac{\dhom-j}{m}\;:\; j\geq 0\}.$
\end{corollary}

\begin{remark}
    While it is important to note that the notation $\Tr$ in Corollary \ref{zeta_function_analyticity} refers specifically to the Hilbert space trace on the operator ideal $\Lc_1(L_2(M,\nu)),$ we also have that
    \[
        AP^{-s} \in \Psi^{-ms}_H(M)
    \]
    and hence the global trace defined in \eqref{global_trace_on_operators} can also be evaluated on $AP^{-s}.$ In fact these two computations give precisely the same answer, and the global trace $\Tr$ on $\Psi_H(M)$ in \eqref{global_trace_on_operators} coincides with the Hilbert space trace on $\Lc_1(L_2(M,\nu)).$

    The Hilbert space trace depends \emph{a priori} on the choice of density $\nu$ while the global trace does not. However, since the Hilbert space trace $\Tr$ is unitarily invariant the choice of density has no effect on the value of the trace.
\end{remark}

We will be only be concerned with the residue at the largest pole $\frac{\dhom}{m}.$ For brevity we introduce the following notation.
\begin{definition}
    Denote
    \[
        C(A,P) := \frac{m}{\dhom}\mathrm{Res}_{s=\frac{\dhom}{m}} \zeta_{A,P}(s).
    \]
\end{definition}

The critereon for universal measurability in terms of zeta functions immediately implies the following:
\begin{corollary}\label{microlocal_form_of_trace_theorem}
    Let $\varphi$ be a normalised trace on the ideal $\Lc_{1,\infty}.$ For any $A \in \Psi^{0,0}_H(M),$ we have
    \[
        \varphi(AP^{-\frac{\dhom}{m}}) = C(A,P)
    \]
\end{corollary}
\begin{proof}
    Let $V = P^{-\frac{\dhom}{m}},$ and
    \[
        \zeta_{A,V}(s) = \Tr(AV^s),\quad \Re(s)>1.
    \]
    By Corollary \ref{zeta_function_analyticity}, $\zeta_{A,V}$ admits a meromorphic continuation to the half-plane $\Re(s)>0,$ with a simple pole at $s=1$ and residue
    \[
        c = \lim_{s\to 1}(s-1)\zeta_{A,P}(\frac{sm}{\dhom}) = \frac{m}{\dhom}\lim_{s\to \frac{\dhom}{m}} (s-\frac{\dhom}{m})\zeta_{A,P}(s) = C(A,P).
    \]
    The content of \cite[Theorem 9.1.5(a)]{LordSukochevZanin2021} is that this implies $\varphi(AV) = c$ for all normalised traces $\varphi$ on $\Lc_{1,\infty}.$
\end{proof}

\begin{corollary}
    For all $A \in \Psi^{-\dhom,0}_H(M)$ for any any normalised trace $\varphi$ on $\Lc_{1,\infty},$ we have
    \[
        \varphi(A) = \frac{1}{\dhom}\Res_0(A).
    \]
\end{corollary}
\begin{proof}
Let $B = AP^{\frac{\dhom}{m}} \in \Psi^{0,0}_H(M).$
Observe that
\[
    C(B,P) = \frac{1}{\dhom}\Res_0(BP^{-\frac{\dhom}{m}}).
\]
Hence Corollary \ref{microlocal_form_of_trace_theorem} yields
\[
    \varphi(A) = \varphi(BP^{-\frac{\dhom}{m}}) = \frac{1}{\dhom}\Res_0(BP^{-\frac{\dhom}{m}}) = \frac{1}{\dhom}\Res_0(A).
\]
\end{proof}

\section{Dixmier measurability and the trace theorem for log-polyhomogeneous operators}\label{log_poly_proof_section}
In this section we prove Theorem \ref{main_Dixmier}, which is a trace theorem for $k$-log-polyhomogeneous pseudodifferential operators. The main technical tool, proved in Appendix \ref{modulation_appendix}, is the following:
\begin{theorem}\label{diagonal_dixmier}
    Let $\Hc$ be a Hilbert space, let $0\leq V\in \Lc_{1,\infty}(\Hc)$ satisfy
    \[
        \lim_{t\to\infty} t\mu(t,V) = 1.
    \]
    Let $k\geq 0,$ and let $A\in \Bc(\Hc)$ be such that for some $p>1,$ we have
    \[
        AV^{-1/p} \in \Ec_{p,k}(\Hc).
    \]
    If
    \[
        \lim_{s\to 0} s^k\Tr(AV^s) = c
    \]
    then for every Dixmier trace $\Tr_{\omega,k}$ on $\Ec_{1,k},$ we have
    \[
        \Tr_{\omega,k}(A) = \frac{1}{(k+1)!}c.
    \]
\end{theorem}

\begin{theorem}\label{main_Dixmier}
    Let $(M,H)$ be a compact Carnot manifold with homogeneous dimension $\dhom.$ Let $k\geq 1,$ and $A \in \Psi^{-\dhom,k}_H(M).$ For all extended limits $\omega,$ we have
    \[
        \Tr_{\omega,k}(A) = \frac{1}{\dhom(k+1)!}\Res_k(A).
    \]
\end{theorem}
\begin{proof}
    By Theorem \ref{schatten_class_theorem}, $A\in \Ec_{1,k}.$

    Let $P\geq 1$ be an elliptic operator on $M$ of order $m.$ By Dave-Haller's Weyl asymptotics, we have $V:= P^{-\frac{\dhom}{m}}\in \Lc_{1,\infty}$ and $\lim_{n\to\infty} n\mu(n,P^{-1/m}) \neq 0.$ Modifying $P$ by a constant, we may assume that
    \[
        \lim_{n\to\infty} n\mu(n,V) = 1.
    \]
    By Dave-Haller's theorem (Theorem \ref{DH_analyticity_theorem}), the family
    \[
        \{AP^{-\dhom s/m}\}_{s\in \Cplx}
    \]
    is a holomorphic family of log-polyhomogeneous pseudodifferential operators, with order function $\mu(s)=-\dhom-\dhom s,$ and
    \[
        \lim_{s\to 0}s^{k+1}\Tr(AV^s) = \lim_{s\to 0}s^{k+1}\Tr(AP^{-\dhom s/m}) = \frac{1}{\dhom}\Res_k(A).
    \]
    Theorem \ref{diagonal_dixmier} implies
    \[
        \Tr_{\omega,k}(A) = \frac{1}{\dhom(k+1)!}\Res_k(A).
    \]
\end{proof}

Similar to Theorem \ref{main_theorem_e_values_version}, we can restate Theorem \ref{main_Dixmier} in a form that makes no reference to traces.
The following is an immediate consequence of Theorem \ref{main_Dixmier} and Theorem \ref{spectrality}.
\begin{corollary}
    Let $(M,H)$ be a compact Carnot manifold with homogeneous dimension $\dhom.$ Let $k\geq 1,$ and $A \in \Psi^{-\dhom,k}_H(M).$
    For any ordered eigenvalue sequence $\{\lambda(n,A)\}_{n=0}^\infty$ of $T$ we have
    \[
        \sum_{n=0}^N \lambda(n,A) = \frac{1}{\dhom(k+1)!}\Res_k(A)\log(N+2)^{k+1}+o(\log(N+2)^{k+1}),\quad N\to\infty.
    \]
\end{corollary}

\section{Non-compact case for the trace theorem}
Theorem \ref{k_log_main_thm} was proved above for compact manifolds $M.$ This was essential in the proof since we needed the existence of an $H$ elliptic operator $P$ with $P^{-1}$ compact. However, having proved the theorem for compact manifolds the general case can be deduced as follows.
\begin{theorem}
    Let $(M,H)$ be a Carnot manifold, and let $\nu$ be a density on $M.$ Let $k\geq 0$ and $A \in \Psi^{-\dhom,0}_H(M).$ Assume that $A$ is compactly supported on the left in the sense that there exists $f\in C^\infty_c(M)$ such that
    \[
        f(x)A(x,y) = A(x,y)
    \]
    (as an equality of distributions). Then $\mathrm{res}_0(A)$ is compactly supported, and $A$ can be realised as an operator belonging to $\Lc_{1,\infty}(L_2(M,\nu)).$ For any normalised trace $\varphi$ on $\Lc_{1,\infty}$ we have
    \[
        \varphi(A) = \frac{1}{\dhom}\mathrm{Res}_0(A)
    \]
    where $\mathrm{Res}_0(A) = \int_M \mathrm{res}_0(A).$

    Similarly, if $k\geq 1$ and $A \in \Psi^{-\dhom,k}$ is compactly supported on the left in the same sense, then $\mathrm{res}_k(A)$ is compactly supported, $A$ can be realised as an operator beloning to $\Ec_{1,k}(L_2(M,\nu)),$ and for all extended limits we have
    \[
        \Tr_{\omega,k}(A) = \frac{1}{\dhom(k+1)!}\Res_k(A)
    \]
    where $\Res_k(A) = \int_M \mathrm{res}_k(A).$
\end{theorem}
\begin{proof}
    Since the support of $A$ is $r$-proper, the distribution $A(x,y)$ is compactly supported in both variables. Assume without loss of generality that $A$ is supported in $U\times U,$ where $U$ is the domain of a coordinate chart which trivialises each of the bundles $H_1,\ldots,H_N$ in the filtration of $TM.$
    Let $d=\mathrm{dim}(M),$ and let
    \[
        \psi:U\to \mathbb{T}^d
    \]
    be a coordinate chart valued in the $d$-torus $\mathbb{T}^d.$

    Under $\psi,$ $A$ is transplanted to an $r$-fibred distribution $A_{\psi}$ on $\Circ^d\times \Circ^d,$ where $r:\Circ^d\times \Circ^d\to \Circ^d$ is the map $r(x,y) = y.$
    By assumption, $A_{\psi}$ is supported in some $K\times K\subset U\times U,$ where $K\subset U$ is compact.

    Let $V$ be the pullback by $\psi$ on functions, i.e.
    \[
        Vf = f\circ \psi^{-1},\quad f \in C^{\infty}(\psi(U)).
    \]
    Note that $V$ induces a partial isometry
    \[
        V:L_2(\psi(U),\psi_*\nu)\to L_2(M,\nu).
    \]
    with $V^*V = \mathrm{id}$ and $VV^*$ is the projection onto functions supported in $U.$
    Since $A$ is supported in $U\times U$ we have
    \[
        A_{\psi} = V^*AV.
    \]
    In particular, $A_{\psi} \in \Ec_{1,k}(L_2(\psi(U),\psi_*\nu)).$
    For any trace $\varphi$ on $\Ec_{1,k},$
    \[
        \varphi(AVV^*) = \varphi(A_{\psi}).
    \]
    Since $A$ is supported in $U\times U,$ we have $AVV^*=A,$ and hence
    \[
        \varphi(A)=\varphi(A_{\psi})
    \]
    where the trace on the left is in $L_2(M,\nu),$ and on the right is in $L_2(\psi(U),\psi_*\nu).$ The pushforward of $H_1,\ldots,H_N$ to $\psi(U)$ are vector bundles denoted $\psi_* H^1,\cdots,\psi_{*}H^N,$ giving $\psi(U)$ the structure of a Carnot manifold. Since each $\psi_*H^1,\ldots,\psi_*H^N$ are trivial, they can be trivially extended to vector bundles on the whole $\mathbb{T}^d.$ Denote these extensions by $\widetilde{H}^1,\ldots,\widetilde{H}^N.$ We claim that
    \[
        A_{\psi} \in \Psi^{-\dhom,k}_{\widetilde{H}}(\Circ^d).
    \]
    To see this, note that $\psi:U\to \psi(U)$ maps the filtreation $H$ to the filtration $\psi^*H,$ and therefore induces a map
    \[
        \Tb\psi:\Tb_HU\to \Tb_{\widetilde{H}}\psi(U) \subseteq \Tb_{\widetilde{H}}(\Circ^d).
    \]
    Let $\Ab \in \Psib^{-\dhom,k}_H(M)$ be a pseudodifferential kernel with $A=\Ab_1.$ The image of $\Ab$ under $\Tb\psi$ is denoted $\Ab_{\psi},$ and we have
    \[
        (\Ab_{\psi})_1 = A_{\psi}.
    \]
    Because $\psi$ preserves the Carnot structure, the zoom vector field on $\Tb_HM$ pushes forward under $\Tb\psi$ to the zoom vector field on $(\psi(U),\psi_*H).$ The family $\Ab_{\psi}$ therefore belongs to $\Psib^{-\dhom,k}_{\widetilde{H}}(\Circ^d).$
    Hence, $A_{\psi} \in \Psi^{-\dhom,k}_{\widetilde{H}}(\Circ^d).$

    By Theorem \ref{k_log_main_thm}, we have
    \[
        \varphi(A) = \varphi(A_{\psi}) = \frac{1}{\dhom(k+1)!}\Res_k(A_{\psi})
    \]
    where $\varphi$ is an arbitrary normalised trace for $k=0$ and a Dixmier trace for $k>0.$
    The residue is diffeomorphism invariant by definition, so we have $\Res_k(A)=\Res_k(A_{\psi}),$ and this completes the proof.
\end{proof}

\appendix
\section{Characterisation of log-polyhomogeneous functions}\label{characterisation_appendix}
In this section, $\delta_t$ denotes an arbitrary group of dilations on $\Rl^N.$ That is, $\delta_t(e_j) = t^{w_j}e_j$ for all $1\leq j\leq N,$ where $w_j\geq 0.$ Let $Z$ be the differential operator generating $\delta,$ that is
\[
    Z = \sum_{j=1}^N w_jx_j\partial_{x_j}.
\]
This is sometimes called the Euler vector field.
Let $|\cdot|$ be a norm on $\Rl^N$ which is homogeneous with respect to $\delta,$ i.e. $|\delta_t\xi|=t|\xi|.$


\begin{definition}\label{k_log_homog_def}
    Say that $f \in C^\infty(\Rl^N)$ is approximately $k$-log-homogeneous of order $m\in \Cplx$ if
    \[
        (Z-m)^{k+1}f \in \Sc(V).
    \]
\end{definition}
Recall that the approximately homogeneous functions $f$ are the ``approximate eigenfunctions" of the Euler vector field $Z,$ i.e. $(Z-m)f \in \Sc(\mathbb{R}^d).$ The meaning of Definition \ref{k_log_homog_def} is that approximately $k$-log-homogeneous functions are ``approximate generalised eigenfunctions of order $k+1$" of $Z.$

The following lemma is not essentially different from \cite[Proposition 2.2]{Taylor1984}.
\begin{lemma}\label{taylor_lemma}
    Let $f\in C^\infty(\Rl^N)$ be approximately homogeneous\footnote{i.e., approximately $0$-log-homogeneous} of order $m\in \Cplx.$ There exists a unique function $H\in C^\infty(\Rl^N\setminus \{0\}),$ homogeneous of order $m$, and $R \in C^{\infty}(\Rl^N\setminus \{0\})$ which coincides with a Schwartz class function outside a neighbourhood of zero such that $f = H+R.$ Moreover, $H$ and $R$ are given by the formulae
    \[
        H(\xi) := \lim_{t\to\infty} t^{-m}f(\delta_t\xi),\quad \xi\neq 0
    \]
    and
    \[
        R(\xi) = -\int_1^{\infty} t^{-m-1}((Z-m)f)(\delta_t\xi)\,dt.
    \]
\end{lemma}
\begin{proof}
    The uniqueness of the decomposition $f=H+R$ is clear from the formula $H(\xi) = \lim_{t\to\infty} t^{-m}f(\delta_t\xi).$

    Abbreviate $K = (Z-m)f.$ By the fundamental theorem of calculus we have
    \[
        f(\xi)-t^{-m}f(\delta_t\xi) = -\int_1^t s^{-m-1}K(\delta_s\xi)\,ds.
    \]
    Since $K\in \Sc(\Rl^N),$ for any $p\geq 0$ and any multi-index $\alpha,$ there exists a constant $C_{\alpha,p}$ such that
    \[
        |\partial_{\xi}^{\alpha}K(\delta_s\xi)| \leq C_{\alpha,p}s^{[\alpha]}(1+s|\xi|)^{-p}.
    \]
    Hence, choosing $p$ sufficiently large,
    \[
        \int_{1}^{\infty} s^{-\Re(m)-1}|\partial_{\xi}^\alpha K(\delta_s\xi)|\,ds \leq C_{\alpha,p}\int_{1}^{\infty} s^{-\Re(m)+[\alpha]-1}(1+s|\xi|)^{-p}\,ds \lesssim |\xi|^{-p}.
    \]
    Therefore the integral
    \[
        R(\xi) = -\int_1^{\infty} s^{-m-1}K(\delta_s\xi)\,ds
    \]
    defines a smooth function on $\Rl^{N}\setminus \{0\},$ with all derivatives rapidly decaying at infinity, and also the limit
    \[
        \lim_{t\to\infty} f(\xi)-t^{-m}f(\delta_t\xi)
    \]
    exists.
\end{proof}

\begin{corollary}\label{taylor_restated_truncated}
    Let $\phi\in C^\infty_c(\Rl)$ be identically $1$ near zero. A function $f\in C^{\infty}(\Rl^N)$ is essentially homogeneous of order $m\in \Cplx$ if and only if there exists a homogeneous function $H$ of order $m$ and a Schwartz class function $S$ such that
    \[
        f(\xi) = (1-\phi(|\xi|))H(\xi)+S(\xi).
    \]
\end{corollary}
\begin{proof}
    Let $H$ and $R$ be the functions from Lemma \ref{taylor_lemma}. Since $\xi\mapsto \phi(|\xi|)$ is smooth and compactly supported, we have
    \[
        f(\xi) = \phi(|\xi|)f(\xi) + (1-\phi(|\xi|))H(\xi) + (1-\phi(|\xi|))R(\xi).
    \]
    Defining $S(\xi) = \phi(|\xi|)f(\xi)+(1-\phi(|\xi|))R(\xi)$ completes the proof.
\end{proof}

\begin{lemma}\label{taylor_logarithmic}
    Let $k\geq 0.$ A function $f \in C^{\infty}(\Rl^N)$ is $k$-log homogeneous of order $m$ if and only if there exist homogeneous functions $H_0,\ldots,H_k$, all of order $m,$ a Schwartz class function $S$ and a smooth function $\phi\in C^{\infty}_c(\Rl)$ equal to $1$ near zero such that
    \[
        f(\xi) = (1-\phi(|\xi|))\sum_{j=0}^{k} \log(|\xi|)^{j}H_j(\xi) + S(\xi).
    \]
\end{lemma}
\begin{proof}
    It is easy to check by differentiation that if $f$ has the stated form then it is $k$-log-homogeneous. We concentrate on proving the converse.

    We may assume without loss of generality that $\Re(m)<0.$ Indeed, otherwise replace $f$ by
    \[
        \widetilde{f}(\xi) = (1-\phi(|\xi|))|\xi|^{-n}f(\xi)
    \]
    for sufficiently large $n,$ by the Leibniz rule we will have $(Z-m+n)^{k+1}\widetilde{f}\in \Sc(\Rl^N),$ and if $\widetilde{f}$ has the required form, then so does $f.$

    The proof is by induction on $k,$ with the $k=0$ case being Corollary \ref{taylor_restated_truncated}.

    Suppose the assertion is true for $k\geq 0,$ and $f$ is $k+1$-log polyhomogeneous of order $m.$ Let $K = (Z-m)f,$ which is $k$-log-polyhomogeneous. We have
    \[
        f(\xi)-t^{-m}f(\delta_t\xi) = \int_t^1 s^{-m-1}K(\delta_s\xi)\,ds.
    \]
    Since $\Re(m)<0,$ this converges pointwise in $\xi$ as $t\to 0,$ So
    \[
        f(\xi) = \int_0^1 s^{-m-1}K(\delta_s\xi)\,ds.
    \]
    Using the inductive hypothesis on $K,$ there exists a Schwartz-class function $S$ and homogeneous functions $H_0,\ldots,H_k$ such that
    \begin{align*}
        f(\xi) &= \sum_{j=0}^k \int_0^1 s^{-m-1}(1-\phi(s|\xi|))\log(s|\xi|)^jH_j(\delta_s\xi)\,ds + \int_0^1 s^{-m-1}S(\delta_s\xi)\,ds\\
               &= \sum_{j=0}^k \int_0^1 s^{-1}(1-\phi(s|\xi|))\log(s|\xi|)^j\,ds H_j(\xi) + \int_0^1 s^{-m-1}S(\delta_s\xi)\,ds\\
               &= \sum_{j=0}^k \int_0^{|\xi|} s^{-1}(1-\phi(s))\log(s)^j\,ds H_j(\xi) + \int_0^1 s^{-m-1}S(\delta_s\xi)\,ds.
    \end{align*}
    The latter integral defines an approximately homogeneous function of $\xi.$

    Assume for the sake of definiteness that $\phi$ is supported in $(-1,1),$ the general case follows from rescaling. For $|\xi|>1,$ we have
    \[
        \int_0^{|\xi|}s^{-1}(1-\phi(s))\log(s)^j\,ds = \int_0^1 s^{-1}(1-\phi(s))\log(s)^j\,ds + \int_1^{|\xi|} s^{-1}\log(s)^j\,ds
    \]
    The first integral above is a constant, the second is $\frac{1}{j+1}\log(|\xi|)^{j+1}.$ Hence $f(\xi)$ has the required form for $|\xi|>1,$ and since $f$ is smooth at zero, the result follows by a truncation.
\end{proof}

\begin{definition}
    Let $k\geq 0.$ Following Lesch \cite{Lesch1999}, say that a smooth function $f$ is $k$-log-polyhomogeneous of order $m\in \Cplx$ if $f$ admits an asymptotic expansion as $|\xi|\to\infty$
    \[
        f(\xi) \sim \sum_{\ell=0}^{k} \sum_{j=0}^{\infty} \log(|\xi|)^{\ell}f_{j,\ell}(\xi).
    \]
    where $f_{j,\ell}$ is homogeneous of order $m-j.$
\end{definition}
Note that this definition involves a choice of norm $|\cdot|;$ but ultimately we will see that the class of functions is independent of this choice.

The following is analogous to Lemma \ref{integral_representation_general}.\eqref{kth_order_integral_representation}, and the proof is the same.
\begin{lemma}
    Let $\Re(m)<0.$ If $f \in C^\infty(\Rl^N),$ and
    \[
         (Z-m)^{-k-1}f \in C^\infty(\Rl^N)
    \]
    then
    \[
        (Z-m)^{-k-1}f = \frac{1}{k!}\int_0^1 t^{-m-1}(-\log t)^{k}f\circ \delta_t \,dt.
    \]
\end{lemma}

Let $\widetilde{\delta}_t$ be the dilation on $\Rl^N\times \Rl$ defined by
\[
    \widetilde{\delta}_t(\xi,h) = (\delta_t\xi,th).
\]
Let $\widetilde{Z}$ be the vector field generating $\widetilde{\delta}_t.$
\begin{theorem}\label{logpolyhomogeneous_characterisation}
    A smooth function $f$ on $\Rl^N$ is $k$-log-polyhomogeneous of order $m$ if and only if there exists a smooth function $\widetilde{f}$ on $\Rl^{N}\times \Rl$ such that $f(\xi)=\widetilde{f}(\xi,1)$ and $\widetilde{f}$ is approximately $k$-log-homogeneous of order $m$ on $\Rl^{N+1}$ with respect to $\widetilde{\delta}.$ That is,
    \[
        (\widetilde{Z}-m)^{k+1}\widetilde{f} \in \Sc(\Rl^{N+1}).
    \]
\end{theorem}
\begin{proof}
    Suppose that there exists $\widetilde{f}$ which is approximately $k$-log-homogeneous of order $m,$ and $f(\xi) = \widetilde{f}(\xi,1).$ By Lemma \ref{taylor_logarithmic}, we have
    \[
        \widetilde{f}(\xi,h) = \sum_{j=0}^k (1-\phi(|\xi|+h))\log(|\xi|+|h|)^jH_j(\xi,h)+S(\xi,h)
    \]
    for some homogeneous functions $H_0,\ldots,H_k$ and a Schwartz function $S.$ Assume that $\phi$ is supported in $(-1,1),$ so that
    \[
        \widetilde{f}(\xi,1) = \sum_{j=0}^k \log(|\xi|+1)^jH_j(\xi,1)+S(\xi,1).
    \]
    The functions $\xi\mapsto H_j(\xi,1)$ are polyhomogeneous of order $m,$ while we also have
    \[
        \log(|\xi|+1)-\log(|\xi|) \sim \sum_{k=1}^\infty (-1)^{k-1}\frac{1}{k}|\xi|^{-k}
    \]
    in the sense of an asymptotic expansion at $\infty.$ This proves that $\widetilde{f}(\xi,1)$ is $k$-log-polyhomogeneous.

    Conversely, let $f$ be a $k$-log-polyhomogeneous function given by the asymptotic sum
    \[
        f(\xi) \sim \sum_{\ell=0}^k \sum_{j=0}^\infty\log(|\xi|)^\ell a_{j,\ell}(\xi)
    \]
    where $a_{j,\ell}$ is homogeneous of degree $m-j.$ It follows from this form that, up to a truncation at zero, there exist polyhomogeneous functions $f_{\ell}$ such that
    \[
        f(\xi) = \sum_{\ell=0}^k \log(|\xi|)^\ell f_{\ell}(\xi).
    \]
    Hence there exist approximately homogeneous functions $\widetilde{f}_{\ell}$ on $\Rl^{N+1}$ for which $f_{\ell}(\xi) = \widetilde{f}_{\ell}(\xi,1).$ Let
    \[
        \widetilde{f}(\xi,h) = \sum_{\ell=0}^k \log(|\xi|)^{\ell}\widetilde{f}_{\ell}(\xi,h).
    \]
    It now suffices to compute $(\widetilde{Z}-m)^{k+1}\widetilde{f}$ using the Leibniz rule.
\end{proof}

\section{$\log^k$-modulated operators}\label{modulation_appendix}
In this section, we use a modification of the theory of modulated operators described in \cite[Section 7.3]{LordSukochevZanin2021}. See also the very similar theory in \cite{Goffeng-Usachev-jmaa-2020}, which is comparable but has slightly different aims.
Recall that $\Hc$ denotes a Hilbert space, and $\Bc(H)$ is the space of bounded operators on $\Hc.$
\begin{definition}
    Let $A\in \Bc(H).$ Let $0\leq V \in \Lc_{1,\infty}(\Hc)$ have trivial kernel. Say that $A$ is weakly $\log^k$-$V$-modulated if there exists $p>0$ such that
    \[
        AV^{-1/p} \in \Ec_{p,k}.
    \]
\end{definition}
Note that $V^{-1/p}$ is unbounded. When we write $AV^{-1/p}\in \Ec_{p,k},$ what we literally mean is that there exists $B\in \Ec_{p,k}$ such that $A=BV^{1/p}.$

\begin{remark}\label{modulated_implies_weak_type}
Since $V^{1/p}\in \Lc_{p,\infty}=\Ec_{p,0},$ if $A$ is weakly $\log^k$-$V$-modulated then $A\in \Ec_{1,k},$ by \eqref{holder_type}.
\end{remark}

We will not actually need the notion, but for the sake of direct comparison of what we do here with the theory in \cite[Chapter 7]{LordSukochevZanin2021}, we also define strong $\log^k$-$V$-modulation:
\begin{definition}
    Let $A \in \Bc(\Hc),$ and let $0\leq V \in \Lc_{1,\infty}(\Hc).$ Say that $A$ is strongly $\log^k$-$V$-modulated if
    \[
        \limsup_{t\to 0} t^{1/2}(-\log t)^{-k}\|A(1+tV)^{-1}\|_{\Lc_2} < \infty.
    \]
\end{definition}
The strongly $\log^0$-$V$-modulated operators are precisely the $V$-modulated operators from \cite[Definition 7.1.2]{LordSukochevZanin2021}.

\begin{lemma}\label{take_the_inverse}
    Let $0\leq V \in \Lc_{1,\infty}.$ If $A$ is strongly $\log^k$-$V$-modulated, then $A$ is weakly $\log^k$-$V$-modulated.
\end{lemma}
\begin{proof}
    For any operator $T$ and any operator $R$ of rank less than $n,$ we have
    \begin{align*}
        \mu(2n,T) &\leq \mu(n,T-R)\\
                  &\leq n^{-1/2}\|T-R\|_{\Lc_2}.
    \end{align*}
    Let $p>0,$ and take $T=AV^{-1/p},$ $R = T\chi_{(\mu(n,T),\infty)}(V),$ we have
    \[
        \mu(2n,AV^{-1/p}) \leq n^{-1/2}\|AV^{-1/p}\chi_{[0,\mu(n,V)]}(V)\|_{\Lc_2}\leq n^{-1/2}\|AV^{-1/p}\chi_{[0,Cn^{-1}]}(V)\|_{\Lc_2}.
    \]
    Let $2^{\ell} \leq n< 2^{\ell+1}.$ We have
    \begin{align*}
        \|AV^{-1/p}\chi_{[0,Cn^{-1}]}(V)\|_{\Lc_2}^2 &\leq \sum_{j=\ell}^{\infty} \|AV^{-1/p}\chi_{[C2^{-j-1},C2^{-j}]}(V)\|_{\Lc_2}^2\\
                                                     &\lesssim \sum_{j=\ell}^\infty 2^{2j/p}\cdot j^k 2^{-j} = \ell^{2k}2^{2\ell(\frac1p-\frac12)}\\
                                                     &\approx n^{\frac2p-1}\log(n)^{2k}.
    \end{align*}
    Therefore
    \[
        \frac{(n+1)^{1-\frac1p}}{\log(n+2)^k}\mu(2n,AV^{-1/p}) \leq \sup_{0<t<1/2} t^{1/2}(-\log t)^{-k}\|A(1+tV)^{-1}\|_{\Lc_2}.
    \]
\end{proof}
The proof of Lemma \ref{take_the_inverse} implies that $A$ is not only weakly $\log^k$-$V$-modulated, but that $AV^{-1/p}\in \Ec_{\frac{p}{p-1},k}$ for all $p>1.$

Recall that $\lambda(n,A)$ denotes the $(n+1)$th largest eigenvalue of a compact operator $A,$ arranged in non-increasing absolute value.

\begin{theorem}\label{diagonal_formula}
    Let $0\leq V \in \Lc_{1,\infty},$ and let $A$ be weakly $\log^k$-$V$-modulated. Let $\{e_n\}_{n=0}^\infty$ be an ordered eigenbasis for $V.$ We have
    \[
        \sum_{n=0}^N \langle e_n,\Re(A)e_n\rangle = \sum_{n=0}^{N} \lambda(n,\Re(A)) + O(\log(N)^k),\quad N\to\infty.
    \]
\end{theorem}
\begin{proof}
    Let $\{f_{n}\}_{n=0}^\infty$ be an eigenbasis for $\Re(A)$ such that
    \[
        \Re(A)f_n = \lambda(n,\Re(A))f_n,\quad n\geq 0.
    \]
    Let $p_N$ be the projection onto the span of $\{e_j\}_{n\leq N},$ and $q_N$ be the projection onto the span of $\{f_n\}_{n\leq N}.$ We need to show that
    \[
        |\Tr(\Re(A)(p_N-q_N))| = O(\log(N)^k).
    \]
    Let $r_N = p_N\vee q_N.$ We have
    \[
        |\Tr(A(r_N-q_N))| = |\Tr(A(1-q_n)r_N-q_N)| \leq \|A(1-q_N)\|_{\infty}\Tr(r_N-q_N) \leq (2N+2)\mu(N,A).
    \]
    By Remark \ref{modulated_implies_weak_type}, $A \in \Ec_{1,k}$ and hence
    \[
        |\Tr(A^*(r_N-q_N))| = |\Tr(A(r_N-q_N))| = O(\log(N)^k)
    \]
    Therefore
    \[
        |\Tr(\Re(A)(r_N-q_N))| = O(\log(N)^k).
    \]
    Next, let $p>1$ be such that $AV^{-1/p}\in \Ec_{\frac{p}{p-1},k}.$ We have
    \begin{align*}
        |\Tr(A(p_N-r_N))| &= |\Tr((r_N-p_N)A(p_N-r_N))|\\
                          &= |\Tr((r_N-p_N)AV^{-1/p}\cdot V^{1/p}(1-p_N))|\\
                          &\leq \|(r_N-p_N)AV^{-1/p}\|_1\mu(N,V)^{1/p}\\
                          &\lesssim\sum_{n\leq 2N+2} \mu(n,AV^{-1/p})\mu(N,V)^{1/p}\\
                          &\lesssim \log(N)^k.
    \end{align*}
    Also,
    \[
        |\Tr(A^*(p_N-r_N))| = |\Tr(A(p_N-r_N))| = O(\log(N)^k).
    \]
    Hence
    \[
        |\Tr(\Re(A)(p_N-r_N))| = O(\log(N)^k).
    \]
    Therefore
    \[
        |\Tr(\Re(A)(p_N-q_N))| \leq |\Tr(\Re(A)(p_N-r_N))| + |\Tr(\Re(A)(r_N-q_N))| = O(\log(N)^k).
    \]
\end{proof}

\begin{theorem}\label{diagonal_formula_nonselfadjoint}
    Let $0\leq V \in \Lc_{1,\infty},$ and let $A$ be weakly $\log^k$-$V$-modulated. Let $\{e_n\}_{n=0}^\infty$ be an ordered eigenbasis for $V.$ We have
    \[
        \sum_{n=0}^N \langle e_n,Ae_n\rangle = \sum_{n=0}^N \lambda(n,A)+O(\log(N)^k),\quad N\to\infty.
    \]
\end{theorem}
\begin{proof}
    By Theorem \ref{diagonal_formula} applied to $A$ and $iA,$ we have
    \[
        \sum_{n=0}^N \langle e_n,Ae_n\rangle = \sum_{n=0}^N \lambda(n,\Re(A))+i\lambda(n,\Im(A)) + O(\log(N)^k),\quad N\to\infty.
    \]
    By Theorem 5.1.5 of \cite{LordSukochevZanin2021}, since the ideal $\Ec_{1,k}$ is geometrically stable, we have
    \[
        \sum_{n=0}^N \lambda(n,A)-\lambda(n,\Re(A))-i\lambda(n,\Im(A)) = O(\log(N)^k).
    \]
\end{proof}

\begin{lemma}\label{zeta_function_lemma}
    Let $0\leq V \in \Lc_{1,\infty}$ and let $A$ be weakly $\log^k$-$V$-modulated. Assume in addition that
    \[
        \lim_{n\to\infty} (n+1)\mu(n,V)=1.
    \]
    If
    \[
        \lim_{s\to 0} s^{k+1}\Tr(AV^{s}) = c
    \]
    then as $N\to\infty,$ we have
    \[
        \sum_{n=0}^N \lambda(n,A) = \frac{1}{(k+1)!}c\log(N+2)^{k+1}+o(\log(N+2)^{k+1}).
    \]
    In particular, for any extended limit $\omega$ we have
    \[
        \Tr_{\omega,k}(A) = \frac{1}{(k+1)!}c.
    \]
\end{lemma}
\begin{proof}
    Let $\{e_n\}_{n=0}^\infty$ be an ordered eigenbasis for $V.$ We have
    \[
        \Tr(AV^s) = \sum_{n=0}^\infty \langle e_n,Ae_n\rangle \mu(n,V)^s = \int_0^{\infty} e^{-s\lambda}\,dS(\lambda)
    \]
    where
    \[
        S(\lambda) = \sum_{\mu(n,V)\geq e^{-\lambda}} \langle e_n,Ae_n\rangle.
    \]
    By the Hardy-Littlewood Tauberian theorem \cite[Theorem I.15.1]{Korevaar-tauberian-2004}, the assumption $\Tr(AV^s)\sim cs^{-k-1}$ implies
    \[
        S(\lambda) \sim \frac{1}{(k+1)!}\lambda^{k+1},\quad \lambda\to\infty.
    \]
    That is, as $\varepsilon \to 0,$ we have
    \[
        \sum_{\mu(n,V)\geq \varepsilon} \langle e_n,Ae_n\rangle = \frac{1}{(k+1)!}(-\log\varepsilon)^{k+1}c+o((-\log\varepsilon)^k).
    \]
    Let $\{N_{\ell}\}_{\ell=0}^\infty$ enumerate the jumps in $\mu(V),$ i.e. $N_0=0$ and $N_{\ell+1}$ is the first value of $n$ such that $\mu(n,V)<\mu(N_{\ell},V).$ We have proved
    \[
        \sum_{n=0}^{N_{\ell}-1} \langle e_n,Ae_n\rangle = \frac{1}{(k+1)!}(-\log\mu(N_\ell-1,V))^{k+1}c+o((-\log\mu(N_{\ell}-1,V)^{k+1})).
    \]
    Since $A$ is weakly $\log^k$-$V$-modulated, Theorem \ref{diagonal_formula_nonselfadjoint} gives
    \[
        \sum_{n=0}^{N_{\ell}-1} \lambda(n,A) = \frac{1}{(k+1)!}(-\log\mu(N_{\ell}-1,V))^{k+1}c+o((-\log\mu(N_{\ell}-1,V))^{k+1}).
    \]
    If $N_{\ell}\leq N < N_{\ell+1},$ we have
    \begin{align*}
        |\sum_{n=0}^{N}\lambda(n,A)-\sum_{n=0}^{N_{\ell}-1} \lambda(n,A)\rangle| &\leq \sum_{n=N_{\ell}}^{N_{\ell+1}-1} \mu(n,A)\\
        &= \sum_{n=N_{\ell}}^{N_{\ell+1}-1} \frac{\log(n)^k}{n}\\
        &\lesssim \log(N_{\ell+1})^{k+1}-\log(N_\ell)^{k+1}.
    \end{align*}
    Due to our assumption that $\lim_{n\to\infty} n\mu(n,V)=1,$ we have
    \[
        \lim_{\ell\to\infty} \frac{N_{\ell+1}}{N_\ell} = 1
    \]
    and therefore
    \[
        |\sum_{n=0}^{N}\lambda(n,A)-\sum_{n=0}^{N_{\ell}-1} \lambda(n,A)\rangle| \leq o(\log(N_{\ell+1})^{k+1}).
    \]
    Finally this means
    \[
        \sum_{n=0}^N \lambda(n,A) = \frac{1}{(k+1)!}(-\log\mu(N,V))^{k+1}c + o((-\log\mu(N,V))^{k+1})
    \]
    as $N\to\infty.$ Again using the assumption that $\mu(N,V)\sim N^{-1},$ we conclude that
    \[
        \sum_{n=0}^N \lambda(n,A) = \frac{1}{(k+1)!}\log(N)^kc+o(\log(N)^{k+1}).
    \]
    By Theorem \ref{spectrality}, this implies that $\Tr_{\omega,k}(T)=\frac{1}{(k+1)!}c$ for all extended limits $\omega.$
\end{proof}

The preceding theorem completes the proof of Theorem \ref{diagonal_dixmier}.

\end{document}